\documentclass[pdf]{article}

\usepackage[english]{babel}

\usepackage[utf8]{inputenc}
\usepackage[utf8]{inputenc}
\usepackage{amsmath}
\usepackage{graphicx}
\usepackage{amssymb}
\usepackage{amsthm}
\usepackage{tikz-cd}
\usepackage{mathrsfs}
\usepackage[colorinlistoftodos]{todonotes}
\usepackage{enumitem}
\usepackage{yfonts}
\usepackage{ dsfont }
\usepackage{MnSymbol}

\title{On collapsing Calabi-Yau fibrations }

\author{Yang Li}

\date{\today}
\newtheorem{thm}{Theorem}[section]
\newtheorem{lem}[thm]{Lemma}
\newtheorem{cor}[thm]{Corollary}
\newtheorem{prop}[thm]{Proposition}

\theoremstyle{definition}

\newtheorem{eg}[thm]{Example}

\newtheorem{conj}[thm]{Conjecture}

\newtheorem*{rmk}{Remark}

\newtheorem{statement}[thm]{Statement}
\newtheorem*{Acknowledgement}{Acknowledgement}

\newcommand{\ie}{\emph{i.e.} }
\newcommand{\cf}{\emph{cf.} }

\newcommand{\C}{\mathbb{C}}
\newcommand{\Z}{\mathbb{Z}}

\newcommand{\norm}[1]{\left\lVert#1\right\rVert}
\newcommand{\Lap}{\Delta}

\DeclareMathOperator{\Hom}{Hom}

\DeclareMathOperator{\Tr}{Tr}

\begin{document}
\maketitle

\begin{abstract}
We develop some techniques to study the adiabatic limiting behaviour of Calabi-Yau metrics on the total space of a fibration, and obtain strong control near the singular fibres by imposing restrictions on the singularity types. We prove a uniform lower bound on the metric up to the singular fibre, under fairly general hypotheses. Assuming a result in pluripotential theory, we prove a uniform fibre diameter bound for a Lefschetz K3 fibred Calabi-Yau 3-fold, which reduces the study of the collapsing metric to a locally non-collapsed situation, and we identify the Gromov-Hausdorff limit of the rescaled neighbourhood of the singular fibre.
\end{abstract}


\section{Introduction}

The present paper studies the adiabatic limiting behaviour of Ricci flat K\"{a}hler metrics on a Calabi-Yau manifold under the degeneration of the K\"{a}hler class. More precisely, let $(X,\omega_X)$ be a compact $n$-dimensional  K\"{a}hler manifold with a holomorphic volume form $\Omega$. Let $f: X\to Y$ be a holomorphic fibration to an $m$-dimensional base manifold $Y$ where $m<n$, with connected fibres denoted by $X_y$ for $y\in Y$. By the adjunction formula, the smooth fibres are automatically Calabi-Yau manifolds in their own right. We can normalise the fibre volume of $\omega_X$ by $\int_{X_y} \omega_X^{n-m}=1$. Away from the set of critical values $S\subset Y$, the fibration is a smooth submersion. Denote by $\omega_0$ the pullback of a K\"{a}hler metric on $Y$, and consider for each $0<t\leq 1$ the unique Calabi-Yau metric $\tilde{\omega}_t$ in the K\"{a}hler class defined by $\omega_t=\omega_0+t\omega_X$. We are interested in uniform estimates for  $\tilde{\omega}_t$ in terms of the geometry of the fibration, and the asymptotic behaviour as $t\to 0$. The prototype example is a K3 fibration over $\mathbb{P}^1$ with at worst nodal fibres.


This question has been investigated by Tosatti in \cite{To}
and subsequently improved by Tosatti et al \cite{TosattiWeinkoveYang} \cite{TosattiZhang},
 where the following result is proven:

\begin{thm}\label{Tosattireview}
Over any compact subset in $Y\setminus S$, the Calabi-Yau metric $\tilde{\omega_t}$ is uniformly equivalent to $\omega_0+t\omega_X$, with constants independent of $t$. As $t\to 0$, with suitable normalisation the K\"ahler potential of $\tilde{\omega}_t$ with respect to  $\omega_0+t\omega_X$ converges in $C^{1,1}(X,\omega_X)$, and the K\"ahler form converges in $C^0$ to the pullback of the so called generalised K\"ahler Einstein metric $\tilde{\omega}_0$ on the base (\cf section \ref{GHconvergencetobase}). Morever, for $y\in Y\setminus S$, the fibrewise metric $\frac{1}{t} \tilde{\omega}_t|_{X_y}$ converges smoothly to the Calabi-Yau metric on the fibre $X_y$ in the class defined by $[\omega_X|_{X_y}]$.
\end{thm}

The uniform constants in \cite{To} blow up doubly exponentially near a singular fibre. Our first main result is that under more restrictive hypotheses about the nature of singularity for the fibration, a lower bound for the Calabi-Yau metric remains uniformly true even near the singular fibres. This requires some modifications to Tosatti's methods. 

\begin{thm}\label{lowerboundonmetrictheorem}(\cf section \ref{uniformlowerboundsection})
If the pushforward of the volume measure  $i^{n^2}\Omega\wedge \overline{\Omega}$ is pointwise bounded by a constant multiple of $\omega_Y^m$,\ie,
\begin{equation}\label{volumecondition}
f_*(i^{n^2}\Omega\wedge \overline{\Omega})\leq C\omega_Y^m,
\end{equation}
 then \begin{equation}\label{lowerboundonCalabiYaumetric}
\tilde{\omega_t} \geq C \omega_t,
\end{equation} 
where the constant is uniform in $t$ and on $X$.
\end{thm}


\begin{rmk}
The uniform upper bound is not true, because otherwise $\tilde{\omega_t}$ will be uniformly equivalent to $\omega_0+t\omega_X$, but the Calabi-Yau volume form $\tilde{\omega_t}^n$ is in general much larger than $\omega_t^n$ at the 
critical points of the fibration.
\end{rmk}

We shall assume this volume condition (\ref{volumecondition}) for most of the work. It is true for example for Lefschetz fibrations with $n\geq 3$, but not for elliptic fibrations with $I_1$ singularities.
The lower bound on the metric coupled with the knowledge of volume form gives also an upper bound on the metric, but this bound blows up near the critical point, and in particular is unable to imply a good bound on the diameter of the fibres. We attempt to overcome this by developing some general tools to bound the distance function of K\"ahler metrics from knowledge on the volume density. 
As a particular consequence,

\begin{thm}
On a fixed compact K\"ahler manifold with a given background metric $\omega$, if we solve the Monge-Amp\`ere equation with a suitably normalised smooth density function $h$,
\[
\omega'=\omega+\sqrt{-1}\partial\bar{\partial} \psi, \quad \omega'^n=h\omega^n,
\]
then the diameter of the new metric $\omega'$ can be bounded in terms of an $L^p$ bound on $h$, for any $p>1$.
\end{thm}

We want to apply this idea to control the rescaled fibre metrics $\frac{1}{t}\tilde{\omega}_t|_{X_y}$, up to the singular fibre. We shall make use of a very special case of the following conjectural statement, which is an extension of Theorem A* in \cite{Ko} and is essentially conjectured in \cite{Ko}. 

\begin{statement}\label{Holderboundpluripotentialtheory}(See \cite{Ko}, discussion after Theorem A*)
Given a deformation family of smooth compact $n$-dimensional K\"ahler manifolds $(\mathcal{X}_y, \omega_y)$ over a small disc, consider a family of classes $[\theta_y]$ on the fibres, which have smooth semi-positive representing forms $\theta_y$,  with uniformly $C^\infty$ bounded local potentials. If we solve the Monge-Amp\`ere equation in the class $[\theta_y]$,
\begin{equation}
\theta_y'=\theta_y+\sqrt{-1}\partial\bar{\partial} \psi_y, \quad \theta_y'^n=h_y\theta_y^n, \quad \sup_{\mathcal{X}_y} \psi_y=0,
\end{equation}
then for $p>1$, a uniform $L^p$ bound on the volume density $h_y$ implies a uniform H\"older estimate on the K\"ahler potential $\psi_y$. The H\"older exponent $\alpha$ can be taken as $0<\alpha<\frac{2}{1+nq}$, where $q$ is the conjugate exponent of $p$.	
\end{statement} 

\begin{rmk}
We warn the reader again that this statement is not proven to date.	
\end{rmk}

Using this and some special algebro-geometric facts in dimension 3, we show
\begin{thm}(\cf section \ref{Generalboundsondistancefunctions}) Assume statement \ref{Holderboundpluripotentialtheory}.
In the case where $f$ is a Lefschetz fibration and $n=3$, the following fibre diameter bound holds uniformly for all $t$ and all fibres:
\begin{equation}\label{uniformfibrediameterbound1}
\text{diam}( {\frac{1}{t}\tilde{\omega}_t|_{X_y}   }  ) \leq C.
\end{equation}
\end{thm}

This uniform fibre diameter bound implies that the neighbourhood of any fibre with the metric $\frac{1}{t}\tilde{\omega}_t$ satisfies a uniform local non-collapsing condition, which links the work to Gromov-Hausdorff convergence theory, and enables the application of many standard geometric inequalities. It is central to everything which follows.

\begin{rmk}
The author expects that (\ref{uniformfibrediameterbound1}) holds more generally in the Lefschetz fibration case with $n> 3$. The difficulty of the general case lies in the pluripotential theoretic problem of achieving uniform H\"older type bounds on the K\"ahler potential near the singular fibre.
\end{rmk}


As a particular application of the techniques developed in this paper, assuming statement \ref{Holderboundpluripotentialtheory}, we obtain strong control on the collapsing Calabi-Yau metric for a Lefschetz K3 fibration. The proof with more general conditions is spread out in the paper.

\begin{thm} Assume statement \ref{Holderboundpluripotentialtheory}.
In the case $f$ is a Lefschetz K3 fibration with at worst nodal fibres, we have
\begin{itemize}
\item The fibrewise metric $\frac{1}{t}\tilde{\omega}_t|_{X_y}$ converges smoothly to the Calabi-Yau metrics on $X_y$,
away from the critical points but not necessarily away from the singular fibres. (\cf section \ref{smoothboundawayfromsingularity} and \ref{convergenceoveronedimensionalbasesection} )
	
\item The neighbourhood of any given fibre with the rescaled  metric $\frac{1}{t}\tilde{\omega}_t$ satisfies a uniform local non-collapsing condition, so we can take the non-collapsed pointed Gromov-Hausdorff limit around any given point $P$, as $t\to 0$. (\cf section \ref{Diameterestimatessection})

\item For the singular fibre $y=f(P)\in S$, the above Gromov-Hausdorff limit is the product metric on $X_y\times \C$, where on the $X_y$ factor we use the orbifold Calabi-Yau metric, and on the $\C$ factor we use a suitably normalised Euclidean metric. (\cf section  \ref{GHlimitaroundsingularfibresection})
	
\item The Gromov-Hausdorff limit of $(X,\tilde{\omega}_t)$ agrees isometrically with the generalised K\"ahler-Einstein metric on $Y$. (\cf section \ref{GHconvergencetobase})
\end{itemize}
\end{thm}

\begin{rmk}
Most of the statements in the above theorem uses only (\ref{volumecondition}), (\ref{lowerboundonCalabiYaumetric}), (\ref{uniformfibrediameterbound1}) and the fact that the base $Y$ has dimension 1.
\end{rmk}

\begin{Acknowledgement}
The author is grateful to his PhD supervisor Simon Donaldson and co-supervisor Mark Haskins for their inspirations, suggestions and encouragements, and to the London School of Geometry and Number Theory (Imperial College London, UCL and KCL) for providing a stimulating research environment. He would also like to thank H-J. Hein and V. Tosatti for comments, and B. Bendtsson, Y. Zhang, E. Di Nezza, S. Dinew for answering some questions.

This work was supported by the Engineering and Physical Sciences Research Council [EP/L015234/1], the EPSRC Centre for Doctoral Training in Geometry and Number Theory (The London School of Geometry and Number Theory), University College London. The author is also funded by Imperial College London for his PhD studies.
\end{Acknowledgement}





\section{The uniform lower bound on the metric}\label{uniformlowerboundsection}

\subsection{The meaning of the volume condition (\ref{volumecondition})}

To explain our volume condition (\ref{volumecondition}), consider a finite coordinate cover of $Y$, where on each coordinate patch we are given a holomorphic $m$-form $dy_1\wedge dy_2 \ldots \wedge dy_m$. We can arrange that these are uniformly equivalent on overlaps of charts. The global holomorphic $n$-form induces by adjunction the fibrewise holomorphic $(n-m)$-forms $\Omega_y$ for $y\in Y$:
\begin{equation}\label{fibrewiseholomorphicvolume}
\Omega=dy_1\wedge dy_2 \ldots \wedge dy_m \wedge \Omega_y.
\end{equation}
Then the condition (\ref{volumecondition}) means that the volumes integrals on the fibres defined by $F_y=\int_{X_y} i^{(n-m)^2}\Omega_y\wedge \overline{\Omega_y}$ are uniformly bounded independent of $y\in Y$. Although $F_y$ implicitly refers to a chart, the boundedness condition depends only on the complex geometry of the fibration.

\begin{rmk} (cf. \cite{To})
Alternatively, one can think of $\Omega_y$ as a trivialisation of the relative canonical line bundle $K_{X/Y}$, and the volume integral $F_y$ defines a pseudonorm on the direct image bundle $f_*K_{X/Y}$. This can fail to be smooth at critical values $y\in S$. Up to normalising factors the curvature form of this pseudonorm is the Weil-Petersson metric on $Y$. Morever, the smooth fibres $X_y$ themselves are Calabi-Yau manifolds with natural volume form proportional to $i^{(n-m)^2}\Omega_y\wedge \overline{\Omega_y}$. We denote by $\omega_{SF,y}$ the unique Calabi-Yau metrics on the fibres cohomologous to the restriction of $\omega_X$, which imposes a volume normalisation, hence $\omega_{SF,y}^{n-m}=\frac{1}{F_y}i^{(n-m)^2}\Omega_y\wedge \overline{\Omega_y}.$ The notation stands for the terminology `semiflat'. As a word of caution, $\omega_{SF,y}$ makes sense for smooth fibres, or normal projective Calabi-Yau variety fibres (\cf \cite{EGZ}), but not on fibres with more severe singularities.
\end{rmk}

\begin{rmk}
According to the appendix of \cite{Zhang}, the condition (\ref{volumecondition}) is true for a one parameter projective family 
of Calabi-Yau varieties with generic smooth fibres, with a nonvanishing section of the relative canonical bundle. 
\end{rmk}

\begin{eg}\label{ODPexample1}
Consider the case where the only singularities in the fibration are ordinary double points (ODP); these are known as Lefschetz fibrations. Away from the singular points, it is clear that the contribution to $F_y$ can be controlled uniformly. Now take for a local model the fibration \[f:  B(1)\subset \C^n \to \C, f(z)=\sum_{i=1}^{n}z_i^2.\] We wish to see if $\int_{B(1)\cap \{f=y\}} \Omega_y\wedge \overline{ \Omega_y}$ is uniformly bounded in $y$. For $n=2$, it is easy to see this diverges logarithmically in $y$. Now let $n\geq 3$.

Let $\omega_{Eucl}=\sum_{1}^{n} \sqrt{-1}dz_i \wedge \bar{dz_i}$, then up to numerical factor $\Omega_y\wedge \overline{ \Omega_y}\sim\frac{1}{H}\omega_{Eucl}^{n-1}|_{X_y}$, where $H=\sum |z_i|^2$. Let $\chi$ be a cutoff function supported in $B(2)$ which is one on $B(1)\subset \C^n$. It is enough to estimate $\int_{B(2)\cap \{f=y\}} \chi \frac{1}{H}\omega_{Eucl}^{n-1}$. By Poincar\'e-Lelong formula, the current of integration on the fibre $X_y$ is $\frac{\sqrt{-1}}{2\pi}\partial \bar{\partial} \log |f-y|$. So up to numerical factor, we can rewrite the integral as $\int_{B(2)}\frac{\chi}{H} \partial \bar{\partial} \log |f-y|  \omega_{Eucl}^{n-1}$. Integrate by part and ignore all smooth terms coming from differentiating $\chi$, using also $\partial \bar{\partial} \frac{1}{H}\omega_{Eucl}^{n-1} \sim \frac{1}{H^2}\omega_{Eucl}^n$, we can rewrite the integral as $\int_{B(2)} \frac{\chi}{H^2} \log |f-y|\omega_{Eucl}^n.$ Now since $\frac{1}{H^{5/2}}$ is $L^1$ integrable for $n\geq 3$, we see the uniform boundedness of the integral. Thus $F_y$ is bounded for $n\geq 3$. 

It follows readily from the same argument that for any fixed $p<n-1$, $\int_{X_y} \frac{1}{H^p} \omega_{Eucl}^{n-1} \leq C$ uniformly in $y$. Also, the integral $F_y$ depends continuously on $y$. 
\end{eg}

One useful property for us later is 
\begin{prop}
(\cite{Zhang}, Theorem 2.1)  The diameter of the fibrewise Calabi-Yau metric $\omega_{SF,y}$, which has volume integral $\int_{X_y}\omega_{SF,y}^{n-m}=1$, can be bounded in terms of $F_y$:
\begin{equation}\label{diameterforsemiflat}
\text{diam}(X_y, \omega_{SF,y}) \leq 2+C \int_{X_y} i^{(n-m)^2}\Omega_y \wedge \overline{\Omega_y},
\end{equation}
where $C$ is independent of $y$. 
\end{prop}
\begin{eg}
For a K3 fibration with ODP degeneration, the fibrewise Calabi-Yau metric $\omega_{SF,y}$ has uniform diameter bound. In contrast, this is not true for
an elliptic fibration with nodal singularity.
\end{eg}

\subsection{The proof of Theorem \ref{lowerboundonmetrictheorem}}

We can write $\tilde{\omega_t}=\omega_t+dd^c\phi$, for some K\"{a}hler potential $\phi$. Our convention is $d^c=\frac{\sqrt{-1}}{2}(\bar{\partial}-\partial)$, so $dd^c=\sqrt{-1}\partial \bar{\partial}$. The Calabi-Yau condition can be written as 
\begin{equation}\label{CalabiYaucondition}
\tilde{\omega_t}^n=a_t t^{n-m} i^{n^2}\Omega\wedge \overline{\Omega},
\end{equation}
where $a_t$ is a cohomological constant:
\begin{equation}
a_t=\frac{1}{ \int_X i^{n^2}\Omega\wedge \overline{\Omega}      }\sum_{k=0}^{m} [\omega_0]^k[\omega_X]^{n-k} {n\choose k} t^{m-k}.
\end{equation}
In the limit $a_t$ converges to \[
a_0= \frac{1}{ \int_X i^{n^2}\Omega\wedge \overline{\Omega} }   [\omega_0]^m[\omega_X]^{n-m} {n\choose m}= \frac{1}{ \int_X i^{n^2}\Omega\wedge \overline{\Omega} }   \langle [\omega_Y]^m, [Y]\rangle {n\choose m} .\]

The following result is known from the theory of the Monge-Amp\`ere equation:

\begin{prop}
(\cite{EGZ}, \cite{EGZ2}, \cite{De-Pa}) There is a uniform constant $C$, such that
$
\norm{\phi}_{L^\infty} \leq C.
$
\end{prop}

Using a Chern-Lu type argument as in \cite{To}, the potential estimate leads to
\begin{prop}\label{omega0isbounded} (\cf \cite{To}, Lemma 3.1)
There is a uniform bound $\Tr_{\tilde{\omega_t}}\omega_0 \leq C$.
\end{prop}

Now we interpret this as a lower bound on the horizontal component of $\tilde{\omega}_t$. Since $\tilde{\omega}_t$ has a prescribed volume form, we infer that the fibrewise restriction of $\tilde{\omega_t}$ has a volume upper bound, under condition (\ref{volumecondition}):
\begin{equation}\label{volumeupperboundonfibre1}
\frac{\tilde{\omega_t}^{n-m}|_{X_y} }{ \omega_{SF,y}^{n-m}} =\frac{\tilde{\omega_t}^{n-m}\omega_0^m }{ \omega_{SF,y}^{n-m}\omega_0^m}\leq C \frac{\tilde{\omega_t}^{n}(\Tr_{\tilde{\omega_t}}\omega_0)^m }{i^{(n-m)^2}\Omega_y\wedge \overline{\Omega_y}\omega_0^m} \leq Ct^{n-m}.
\end{equation}
Observe also that 
\begin{equation}\label{volumeupperboundonfibre2}
\frac{{\omega_t}^{n-m}|_{X_y} }{ \omega_{SF,y}^{n-m}}\leq Ct^{n-m}.
\end{equation}

Define the oscillation to be $\text{osc}=\sup-\inf$.
\begin{lem} Under condition (\ref{volumecondition}),
The fibrewise oscillation of the K\"{a}hler potential $\phi$ satisfies the uniform bound
\begin{equation}\label{fibrepotentialoscillation}
\text{osc}_{X_y} \phi \leq Ct.
\end{equation}
\end{lem}
\begin{proof}
Notice first that on each smooth fibre $X_y$, the forms $\omega_X$, $\omega_{SF,y}$ and $\frac{1}{t}\tilde{\omega_t}$ all belong to the same cohomology class. We can write
\[
\omega_X=\omega_{SF,y}+dd^c (v-\phi/t),\text{ and } \frac{1}{t}\tilde{\omega_t}=\omega_{SF,y}+dd^c v .
\]
Now consider the Monge-Amp\`ere equation satisfied by $\omega_X$ and $\frac{1}{t}\tilde{\omega_t}$, with the background metric $\omega_{SF,y}$. The inequalities (\ref{volumeupperboundonfibre1}) and (\ref{volumeupperboundonfibre2}) imply $L^\infty$ control on the volume form. Notice that since these $\omega_{SF,y}$ are Ricci flat, with unit volume and uniformly bounded diameters by (\ref{diameterforsemiflat}), they have uniformly bounded Sobolev and Poincar\'e constants, and so Yau's $L^\infty$ bound applies to give
\[
\text{osc}_{X_y} v \leq C, \text{ and } \text{osc}_{X_y}(v-\phi/t) \leq C,
\]
hence the conclusion holds on all smooth fibres, and therefore extends by continuity to all fibres. This concludes the proof.

For later use, let us observe further that, a step in Aubin-Yau's standard $L^\infty$ estimate uses integration by part to bound the $L^2$ gradient of the potential in terms of zeroth order quantities. From that we clearly see \begin{equation}
\int_{X_y} |\nabla \frac{1}{t} \phi|^2 \omega_{SF,y}^{n-m} \leq C.
\end{equation}
\end{proof}

Now we follow \cite{To} to define the fibrewise average function of $\phi$:
\begin{equation}
\underline{\phi}=\int_{X_y} \phi \omega_X^{n-m}.
\end{equation}
The current of integration for the fibres $X_y$ is continous in $y$ in the weak topology, so $\underline{\phi}$ is a continous function on $Y$. In general we do not expect it to be smooth at the critical values of the fibration, so we need a little extra care.

We compute using that $\omega_t $ is uniformly equivalent to $\omega_0+t\omega_X$:
\[
dd^c \underline{\phi}=f_* (dd^c \phi \wedge \omega_X^{n-m})
\geq -f_*(\omega_t\wedge \omega_X^{n-m}) \geq -C(\omega_Y+ t f_*( \omega_X^{n-m+1}))
\]
where $f_*$ denotes the pushforward of forms, and everything is interpreted in the distributional sense.
Notice also that under condition (\ref{volumecondition}),
\[
\begin{split}
f_*(\omega_X^{n-m+1})\leq \frac{\omega_Y^{m-1}\wedge f_*(\omega_X^{n-m+1})}{\omega_Y^m} \omega_Y=&\frac{f_*(\omega_0^{m-1}\wedge \omega_X^{n-m+1})}{\omega_Y^m}\omega_Y   \\
&\leq C\frac{f_*( \omega_X^{n})}{\omega_Y^m}\omega_Y \leq C\omega_Y.
\end{split}
\]
Hence $dd^c\underline{\phi} \geq -C\omega_Y$ in the distributional sense. Plurisubharmonicity is preserved under pull back to $X$. We identify $\underline{\phi}$ with its pullback to $X$.
Then taking the trace with respect to $\tilde{\omega_t}$, we get 
\[
\Lap_{\tilde{\omega_t}} \underline{\phi} \geq -C\Tr_{\tilde{\omega_t}}\omega_0\geq -C.
\]
Since $\phi$ is bounded in $L^\infty$, so is $\underline{\phi}$, hence $\underline{\phi}$ must be in $L^2_1$ on X.

Now we can prove the lower bound (\ref{lowerboundonCalabiYaumetric}) in Theorem \ref{lowerboundonmetrictheorem} as claimed:
\begin{proof}(Theorem \ref{lowerboundonmetrictheorem})

Apply the Chern-Lu formula to get 
\[
\Lap_{\tilde{\omega_t}} \log \Tr_{\tilde{\omega_t} } \omega_X \geq -C \Tr_{\tilde{\omega_t} } \omega_X.
\]
Now essentially the same computation as in \cite{To} gives
\[
\Lap_{\tilde{\omega_t}} (\log \Tr_{\tilde{\omega_t} }\omega_X -\frac{C}{t }(\phi-\underline{\phi}) ) \geq  \Tr_{\tilde{\omega_t} } \omega_X-\frac{Const}{t}.
\]
The rest of the proof is the classical maximum principle for quasilinear ellitpic equations in the weak formulation.
Let $u=\log \Tr_{\tilde{\omega_t} }\omega_X -\frac{C}{t }(\phi-\underline{\phi})$, then by the bound on the fibrewise oscillation of $\phi$, we have
\begin{equation}
\Lap_{\tilde{\omega_t}} u\
\geq C(e^u-e^{C'-\log t}),
\end{equation}
for some uniform positive constant $C, C'$, where everything is interpreted in the distributional sense, and we know a priori that $u$ is continuous and $u\in L^2_1$. Now multiply the above inequality by the test function $(u-C'+\log t)_+$, and integrate by part. We get
\[
\int_{u>C'-\log t} |\nabla u|_{\tilde{\omega_t}}^2 {\tilde{\omega_t}}^n + \int_{u>C'-\log t} C(e^u-e^{C'-\log t})(u-C'+\log t) {\tilde{\omega_t}}^n \leq 0,
\]
from which we see that the set with $u>C'-\log t$ has zero measure. So in fact $u\leq C'-\log t$, hence $\Tr_{\tilde{\omega_t} }\omega_X \leq C/t. 
$
This combines with the estimate $\Tr_{\tilde{\omega_t}}\omega_0 \leq C$ to give $\Tr_{\tilde{\omega_t}}\omega_t \leq C$, as desired.
\end{proof}

An argument which appeared in \cite{To} now gives an upper bound of $\tilde{\omega_t}$. Notice that this provides uniform control on $\tilde{\omega_t}$ away from the critical points of the fibration, but not necessarily away from the singular fibre.
\begin{cor}\label{upperboundonmetric}
(Upper bound) Let $H=\frac{\omega_X^{n-m}\omega_0^m}{\omega_X^n}$, then under condition (\ref{volumecondition}), we have a uniform constant $C$, such that
\begin{equation}
\tilde{\omega_t} \leq \frac{C}{H}\omega_t.
\end{equation} 
\end{cor}
\begin{proof}
By simultaneous diagonalisation at a point, one shows 
\[\Tr_{\omega_t}\tilde{\omega_t} \leq (\Tr_{\tilde{\omega_t}}{\omega_t})^{n-1}\frac{\tilde{\omega_t}^n} {\omega_t^n}. \]
The volume ratio $\frac{\tilde{\omega_t}^n} {\omega_t^n}\leq \frac{C}{H}$, so $\Tr_{\tilde{\omega_t}}{\omega_t}\leq C$ implies $ \Tr_{\omega_t}\tilde{\omega_t} \leq \frac{C}{H}$.
\end{proof}

\begin{eg}
In the ODP situation as in example \ref{ODPexample1}, $H\sim \sum |z_i|^2$. Notice that the upper bound on $\tilde{\omega}_t$ says $\frac{1}{t} \tilde{\omega}_t|_{X_y} \leq \frac{C}{H}\omega_X$, so the distance function on $(X_y, \frac{1}{t} \tilde{\omega}_t|_{X_y} )$ can be estimated above by line integrals of the shape $\int \frac{1}{H^{1/2}}\sim -\log |y|$, from which we get a logarithmic diameter bound \begin{equation}\label{logarithmicdiameterbound}
\text{diam}(\frac{1}{t} \tilde{\omega}_t|_{X_y}   )\leq -C \log |y|.
\end{equation}
\end{eg}

\section{Smooth control away from critical points}\label{smoothcontrolawayfromcriticalpoints}

In this section let us work under the volume condition (\ref{volumecondition}). For simplicity we also assume smoothness of the base throughout the paper. We work in a neighbourhood of a point away from the critical points of the fibration $f:X\to Y$, (\ie on the set where the quantity $H$ defined in Corollary \ref{upperboundonmetric} is bounded positively below), but not necessarily away from the singular fibres. (Our discussions below apply equally well to Tosatti's original more general situation in \cite{To}, where no volume condition is made, but then we have to work away from $S$). 

In a neighbourhood of our point, the uniform control $C^{-1}\omega_t \leq \tilde{\omega_t} \leq C\omega_t$ holds. We will exhibit appropriate scalings and coordinates such that $\tilde{\omega_t}$ has smooth bounds. For this, we remark that it is natural to think of our neighbourhood as living inside $X\times \C^m$, where $m=\dim Y$, via the map
\[f_t: B_{{\omega_t}}(P, Ct^{1/2}) \to X \times \C^{m},\] given by $x\mapsto (x, t^{-1/2}(f-f(P)))$.  Here we implicitly used a chart on $Y$ to regard $f$ as mapping to $\C^m$. The target is equipped with a natural metric $\omega_X+\omega_{Eucl}$. The map $f_t$ is clearly a smooth embedding, and the pullback metric is uniformly equivalent to $\frac{1}{t}\omega_t$. 

Now let $z_1, \ldots z_{n-m}$ be holomorphic functions on a fixed neighbourhood of $P$ in $X$, which restrict to a set of local coordinates on the fibre $X_{f(P)}$. Let $u_1, \dots u_m$ be the standard coordinates on $\C^m$. Then the composition 
\[
B_{{\omega_t}}(P, Ct^{1/2}) \to \text{Image}(f_t) \xrightarrow{(z_1, \ldots, u_m)}  \C^n
\]
is also an embedding map. By adjusting the metric on $\C^n$ by a bounded multiple, we may assume $B_{\omega_t}(P,Ct^{1/2})$ contains the unit ball $B(1)\subset \C^n$, and the metric $\frac{1}{t}\omega_t$ is uniformly equivalent to the standard metric on $\C^n$. We can thus regard $\frac{1}{t}\tilde{\omega_t}$ as a metric on $B(1)\subset \C^n$. We can find a local K\"ahler potential for $\frac{1}{t}\tilde{\omega_t}$, which has some uniform $C^{1,1}$ bound in a smaller ball $B(1/2)\subset \C^n$. But we also know the volume of $\frac{1}{t}\tilde{\omega_t}$ is given by a holomorphic $n$-form, so in particular has smooth bounds on $B(1/2)$, by the regularity of holomorphic functions. By Evans-Krylov theory for the Monge-Amp\`ere equation, along with the Schauder theory, (see Siu's lecture notes \cite{Siu} for these standard backgrounds),  we can bootstrap to bounds on all derivatives. Now adjust some constants and we have

\begin{prop}\label{smoothboundawayfromsingularity}
	Assuming the volume condition (\ref{volumecondition}), away from the critical points, \ie $H\geq \text{Const}$, we can find holomorphic coordinates described as above, 
	such that 
	the ball $B_{\frac{1}{t}\tilde{\omega_t} }(P, C)$ can be regarded as a domain in $\C^n$, containing the unit ball in $\C^n$, and contained in a fixed large ball in $\C^n$, and the metric is smoothly equivalent to the Euclidean metric, with smooth bounds. All bounds are independent of $t$ and the position of the point $P$.
\end{prop}

\begin{cor}\label{smoothboundfibrewise}
	In the above setup, we have the uniform bound away from critical points for the fibre metric \begin{equation}
		\norm{\nabla_{\omega_X}^{(k)} \frac{1}{t}\tilde{\omega_t}|_{X_y} }_{L^\infty}\leq C(k).
	\end{equation}
\end{cor}

\begin{cor}
	In the above setup, we have the uniform bound away from the critical points for the fibrewise derivatives of the trace
	\begin{equation}\label{smoothboundTromega0}
		\norm{\nabla_{\omega_X}^{(k)}  (\Tr_ {\tilde{\omega_t} } \omega_0)  |_{X_y} }_{L^\infty}\leq C(k).
	\end{equation}	
	
\end{cor}

\begin{rmk}
The idea of looking at the Calabi-Yau metric $\tilde{\omega_t}$ in a scale where the fibre has constant size is inspired by a talk given by J. Fine. The author thanks V. Tosatti for pointing out the reference \cite{TosattiZhang} Theorem 1.1, which contains essentially the same argument for corollary \ref{smoothboundfibrewise}.
\end{rmk}

\section{Diameter estimates}\label{Diameterestimatessection}


\subsection{General bounds on distance functions}\label{Generalboundsondistancefunctions}

We begin by developing some general tools for estimating diameters and distance functions on a compact K\"{a}hler manifold directly from the K\"ahler potential, which is of independent interest. 






\begin{thm}
	Suppose $M$ is a compact K\"{a}hler manifold of dimension $n$, with a fixed background K\"{a}hler metric $\omega$, and define a new metric $\omega'$ by $\omega'=\omega+dd^c\psi$.
	
\begin{itemize}
\item 
If $\norm{\psi}_{C^\alpha}\leq C$, then the distance function in the new metric $d_{\omega'}(x,y)$ has a $C^{\alpha/2}$ H\"{o}lder bound in terms of the original metric, with constants only depending on $(X,\omega)$ and $\norm{\psi}_{C^\alpha}$. \\

\item  Without any control on $\psi$, then 
\begin{equation}
\int_{M\times M}  d\mu_x d\mu_y \exp(   \frac{ d_{\omega'}(x,y)   }{C}         ) \leq C',
\end{equation}
where $d\mu=\omega^n$ is the reference measure, and $C, C'$ denote constants independent of $\psi$.
\end{itemize}	
\end{thm}

\begin{proof}
Work in a coordinate chart $B(x,R)$ on $M$, where $\omega$ is uniformly equivalent to the Euclidean metric. Consider the distance function on $y\in B(x,R)$ defined by $\rho_x(y)=d_{\omega'}(x,y)$, which makes sense because $\omega'$ is smooth. Then $|\nabla_{\omega'}{\rho_x}|\leq 1.$ Thus
\[
|\nabla_{\omega}{\rho_x}|^2 \leq \Tr_{\omega}\omega'. 
\]
The next idea comes from the Chern-Levine inequality. For $r<R/3$, we can choose a cutoff function $\chi$ supported on $B(2r)$ and equal to one on $B(r)$, such that $dd^c\chi \leq \frac{C}{r^2}\omega$. We first consider the case with some H\"older bound on potential. This gives \[
\begin{split}
\int_{B(r)}(\Tr_{\omega}{\omega'}) \omega^n=n\int_{B(r)} \omega'\omega^{n-1}\leq n\int \chi\omega'\omega^{n-1}\leq Cr^{2n}+n\int \chi dd^c \psi \omega^{n-1}\\
=Cr^{2n}+n\int (\psi-\inf_{B(2r)}{\psi}) dd^c \chi   \omega^{n-1} \leq Cr^{2n}+C (osc_{B(2r)}\psi) r^{2n-2}\leq C r^{2n-2+\alpha}.
\end{split}
\]
Hence $\int_{B(r)} |\nabla_\omega \rho_x|^2 \omega^n \leq Cr^{2n-2+\alpha}$. This means the function $\rho_x$ is bounded in the Morrey norm. In particular by Morrey embedding theorem, and the trivial observation $\rho_x(x)=0$,
\[
\norm{\rho_x}_{C^{\alpha/2} (B(x, R/3))} \leq C,
\]
where $C$ depends only on the $(M,
\omega)$ and bound on $\psi$. This implies
\begin{equation}\label{Holderboundondistance}
d_{\omega'}(x,y) \leq C d_\omega (x,y)^{\alpha/2}
\end{equation}
on coordinate charts. The general H\"older bound for any $x,y\in M$ follows from the triangle inequality. 

For the case with no assumption on the potential, we will use some ideas in the standard proof of Skoda integrability theorem. For convenience, we assume the normalisation $\int_M \omega^n=1$.
We notice that by the monotonicity formula in the theory of Lelong numbers, the estimate
\[
\int_{B(r)} \omega'\omega^{n-1} \leq Cr^{2n-2}\int_{M} \omega'\omega^{n-1} \leq Cr^{2n-2}
\]
holds always. This combined with the John-Nirenberg inequality gives
\begin{equation}
\int_{M} \exp(  \frac{\rho_x-\bar{\rho_x}}{C}   ) \omega^n \leq C',
\end{equation}
where $\bar{\rho_x}$ is the average of $\rho_x$, and the constants are independent of $x$ or $\omega'$. 
Write the measure $d\mu=\omega^n$, and define $\bar{\rho}=\int_M\int_M d_{\omega'} (x,y)d\mu_x d\mu_y$.
Then by Fubini theorem and Jensen convexity inequality,
\[
\begin{split}
\int_M  d\mu_y \exp(   \frac{\bar{\rho_y}-\bar{\rho}}{C}         )      =\int_M d\mu_y \exp( \int_{M}   \frac{d_{\omega'}(x,y)-\bar{\rho_x}}{C}   d\mu_x          )  \\
  \leq \int_{M}\int_M \exp(  \frac{d_{\omega'}(x,y)-\bar{\rho_x}}{C}   ) d\mu_x d\mu_y \leq C'.
\end{split}
\]
It is quite well known (see for example \cite{Zhang}, the proof of lemma 2.2) that \[
\int_M\int_M d_{\omega'}(x,y)^2 d\mu_x d\mu_y \leq C\int_{M} \Tr_\omega \omega' d\mu\leq C,\] so
the number $\bar{\rho}$ is controlled by a constant, using Cauchy-Schwarz.
Thus
\begin{equation}
\int_M  d\mu_y \exp(   \frac{\bar{\rho_y} }{C}         ) \leq C',
\end{equation}
and therefore
\begin{equation}
\begin{split}
\int_{M\times M}  d\mu_x d\mu_y \exp(   \frac{ d_{\omega'}(x,y)   }{C}         )
\leq & (\int_{M\times M} d\mu_x d\mu_y \exp(  \frac{d_{\omega'}(x,y)-\bar{\rho_x}}{C/2}   ) )^{1/2}\\ & ( \int_{M\times M}  d\mu_x d\mu_y \exp (\frac{\bar{\rho_x} }{C/2})      )^{1/2}
\leq C',
\end{split}
\end{equation}
implying the claim.
\end{proof}

\begin{rmk}
	Intuitively, one may think the potential has the same order as the distance squared, which explains the way their H\"older exponents are related.
\end{rmk}

\begin{rmk}
In the argument, the K\"ahler class of the reference metric $\omega$ can be decoupled from the class of $\omega'$. This means if $\omega'=\omega''+\sqrt{-1}\partial \bar{\partial} \psi$, where $0\leq \omega''\leq C \omega$, then the first part of the theorem involving the H\"older bound still holds.
\end{rmk}

The following is an immediate consequence, building on substantial work in pluripotential theory:

\begin{cor}\label{boundedgeometryHolderbound}
	Suppose $M$ is a compact K\"{a}hler manifold of dimension $n$, with a fixed background K\"{a}hler metric $\omega$. Let $h$ be a smooth density function with a bound on $L^p$ for some $p>1$, such that the volume $\int_M h\omega^n=\int_M \omega^{n}$. We consider the unique metric $\omega'=\omega+dd^c\psi$ solving the Monge-Amp\`ere equation $\omega'^n=h\omega^n$. Then the distance function in the new metric $d_{\omega'}(x,y)$ has a H\"{o}lder bound in terms of the original metric, with constants and H\"older exponent only depending on $(X,\omega)$ and $\norm{h}_{L^p}$. 
\end{cor}

\begin{proof}
By theorem A* of \cite{Ko}, under a suitable normalisation, the K\"{a}hler potential is bounded in some H\"{o}lder norm, $\norm{\psi}_{C^\alpha}\leq C$. The constant in \cite{Ko} is independent of the manifold under the assumption of bounded geometry (see \cite{Ko}), which means uniform control on the diameter of $(X,\omega)$, the injectivity radius and the bisectional curvature. The H\"older exponent $\alpha$ can be taken to be any positive number smaller than $\frac{2}{nq+1}$, where $q$ is the conjugate exponent of $p$. 
\end{proof}

\begin{rmk}
The strength of this result can be seen by comparing with the general Riemannian situation, where even the $L^\infty$ bound on the volume form is very far from controlling the diameter. Notice we do not even require any control on the Ricci curvature of $\omega'$.
\end{rmk}

\subsection{The uniform diameter bound}

Now we return to the collapsing metric problem.

\begin{conj}\label{uniformfibrediameterboundconjecture}
In the Lefschetz fibration case with $n\geq 3$, 
there is a uniform diameter bound on the normalised fibre metric, independent of $t$ and the fibre $X_y$: \begin{equation}\label{uniformfibrediameter}
\text{diam}(\frac{1}{t}\tilde{\omega_t}|_{X_y} )\leq C.
\end{equation}
\end{conj}

\begin{rmk}
We restrict attention to Lefschetz fibrations for our modest purpose and for the simplicity of exposition, although one may hope the result is true more generally for smoothings of Calabi-Yau varieties. To motivate this conjecture, recall the logarithmic estimate (\ref{logarithmicdiameterbound}), which is not far from (\ref{uniformfibrediameter}). As another piece of evidence, recall theorem \ref{Tosattireview} of Tosatti et al, which asserts that on all smooth fibres $\frac{1}{t}\tilde{\omega}_t|_{X_y}\to \omega_{SF,y}$ as $t\to 0$, with convergence rate dependent on the fibre. Now the limiting metrics on the fibres have uniform diameter bounds, which is compatible with the conjecture.
\end{rmk}

This uniform diameter bound is a very important condition, as shall be apparent later. In particular this leads to the fact that in a suitable scale, the neighbourhood of any given fibre (which could be a singular fibre) satisfies a local non-collapsing bound as $t\to 0$, and therefore we can apply the well developed machinery of noncollapsing Gromov-Hausdorff convergence theory. More delicate consequences will be treated later.

\begin{cor}
	Assuming the volume condition (\ref{volumecondition}) and the uniform fibre diameter bound (\ref{uniformfibrediameter}), $\frac{1}{t}\tilde{\omega_t}$ satisfies the local volume non-collapsing estimate. There are constants, such that for any central point $P$, and for any $C \leq R\leq C/t^{1/2}$,
	\begin{equation}\label{localvolumenoncollapsing1}
	Vol(B_{\tilde{\omega_t}}(P, Rt^{1/2})) \geq CR^{2m}t^{n}.
	\end{equation}
	Morever, for any $R\leq C$,
	\begin{equation}\label{localvolumenoncollapsing2}
	Vol(B_{\tilde{\omega_t}}(P, Rt^{1/2})) \geq CR^{2n}t^{n}.
	\end{equation}
\end{cor}

\begin{proof}
	Assume first that $R\geq C$.
	We notice as a particular consequence of the condtion (\ref{volumecondition}), on any fibre, there is some region with nontrivial $\omega_X$ measure, where $H\leq C$, so the metric $\tilde{\omega_t}$ is uniformly equivalent to $\omega_t$. Thus if $d_{\omega_Y}(y,y')\leq Rt^{1/2}/C$, then the distance between the two fibres $X_{y}$ and $X_{y'}$ is less than $\frac{1}{3}Rt^{1/2}$. Now by the fibrewise diameter bound, we know that each fibre has diameter less than $\frac{1}{3}Rt^{1/2}$. This means we can reach any point on a nearby fibre within a controlled amount of distance, so the ball $B_{\tilde{\omega_t}}(P, Rt^{1/2})$ contains the preimage of $B_{\omega_Y}(f(P), Rt^{1/2}/C)$. Since we know the volume form of $\tilde{\omega_t}$ is $a_t t^{n-m} i^{n^2}\Omega\wedge \overline{\Omega}$, we obtain $Vol ( f^{-1}( B_{\tilde{\omega_t}}(P, Rt^{1/2})   ) \geq CR^{2m}t^{n}$. Hence the estimate (\ref{localvolumenoncollapsing1}) holds. The other estimate (\ref{localvolumenoncollapsing2}) follows from Bishop-Gromov inequality using the Ricci flatness of $\tilde{\omega_t}$.
\end{proof}

We make some preliminary remarks about the nature of the problem to prove the uniform diameter bound.

\begin{rmk} (Local reduction)
	Away from the region with small $H$, the metric $\frac{1}{t} \tilde{\omega_t}|_{X_y} $ has smooth bounds. We also know already that on $X_y$, the oscillation of $\phi$ is bounded by $Ct$, (\cf (\ref{fibrepotentialoscillation})). Another piece of special information, by recalling the proof of (\ref{fibrepotentialoscillation}), is the $L^2$ gradient estiamtes
	\begin{equation}\label{L2gradientestimate}
	\int_{X_y} |\nabla \frac{1}{t} \phi|^2 \omega_{SF,y}^{n-m} \leq C.
	\end{equation}
So after possibly adjusting the average value of $\phi$, we can assume smooth bound on $\frac{1}{t}\phi$ in the $\omega_X|_{X_y}$ metric in the region $H>1$. 
\end{rmk}

Consequently, the following local statement implies the conjecture:
\begin{statement} (Reduction to local case)\label{localreductiontoquadric}
Consider the setup of example \ref{ODPexample1}. Suppose there is a metric $\omega'=\omega_{Eucl}|_{X_y}+i \partial \bar{\partial} \psi$ on $X_y\cap B(1)$, with the volume bound
\[
\omega'^{n-1} \leq \frac{C}{H}\omega_{Eucl}^{n-1}|_{X_y},
\]
and morever $osc_{X_y\cap B(1)} \psi \leq C$, the local version of 
 (\ref{L2gradientestimate}) holds, and $\psi$ has smooth bounds near the boundary, then the diameter of the $\omega'$ metric has a uniform bound for all $|y|\leq 1/2$.
\end{statement}

There is a compact analogue of this local reduction:

\begin{eg}\label{quadricfamily}
Consider the smoothing of a singular quadric in $\mathbb{P}^n$, where $n\geq 3$, and $y$ is a small parameter:
\[ X_y=\{ -yZ_0^2+Z_1^2+\ldots +Z_n^2=0         \}. \]
This can be viewed as a compact model for ODP degeneration. We equip the fibres $X_y$ with the Fubini-Study metric $\omega_{Fubini}$. There is a function $H$ which on the chart  $\{ Z_0=1, \sum_1^n |Z_i|^2<1  \}$ is comparable to $\sum_1^n{|Z_i|^2}$, and is of order one elsewhere. The following special statement about the quadric family implies the conjecture:
\begin{statement}\label{localreductiontoquadric2}
	(Reduction to quadric family)
Suppose there is a metric $\omega'=\omega_{Fubini}|_{X_y}+i \partial \bar{\partial} \psi$ on $X_y$, with the volume bound
\[
\omega'^{n-1} \leq \frac{C}{H}\omega_{Fubini}|_{X_y}^{n-1},
\]
and morever $osc_{X_y} \psi \leq C$, the local version of the gradient estimate (\ref{L2gradientestimate}) holds, and $\psi$ has smooth bounds away from $\{ Z_0=1, \sum_1^n |Z_i|^2<1  \} $, then the diameter of the $\omega'$ metric has a uniform bound for all $|y|\leq 1/2$.
\end{statement}

The reason these are all equivalent, is that one can always isometrically and holomorphically embed the $\omega'$ metric in a small ball of interest into the compact quadric model space using some cutoff functions, and extend $\omega'$ to the whole quadric model with smooth bounds. With these in mind, we can sometimes switch freely between global and local discussions.

\end{eg}

\begin{rmk}
Another observation is that for the purpose of bounding $\frac{1}{t}{\tilde{\omega_t}}|_{X_y}$, the background metric is free for us to choose. In the ODP smoothing example, we have a host of natural background metrics:

\begin{itemize}
	\item the Fubini-Study metric on $X_y$;
	\item the restriction of the Euclidean metric, defined locally;
	\item the Calabi-Yau metric on $X_y$;
	\item the standard Stenzel metric on smoothings of ODP, defined locally.
\end{itemize}
Choosing a good background metric can sometimes be very convenient, which is an idea we shall utilise in the next section.

\end{rmk}


\subsection{The 3-fold case}

Assuming the conjectural Statement \ref{Holderboundpluripotentialtheory}, we  
prove Statement \ref{localreductiontoquadric2}  for the special case $n=3$ using very special algebro-geometric facts,  thus resolving the conjecture \ref{uniformfibrediameterboundconjecture} for threefolds. The application is to a Lefschetz fibration by K3 surfaces.

We first show independently that the diameter bound (\ref{uniformfibrediameter}) is true on the singular central fibre of the Lefschetz fibration.
	
\begin{prop}
On the central fibre $X_{y=0}$, the diameter bound holds uniformly in $t$:
\[
\text{diam} (\frac{1}{t}\tilde{\omega}_t|_{X_0})\leq C.
\]
\end{prop}

\begin{proof}
This follows from the local reduction statement \ref{localreductiontoquadric}. The central fibre is an orbifold with singularity $\C^2/\Z_2$, and the volume density of $\frac{1}{t}\tilde{\omega}_t|_{X_0}$ has $L^\infty$ bound with respect to the flat orbifold metric, so we can pass to the local double cover and apply the discussions in section \ref{Generalboundsondistancefunctions}.
\end{proof}

Next we show the main result of this section, using \ref{Holderboundpluripotentialtheory}. The author thanks S. Dinew for clarifying some pluripotential theoretic issues.

\begin{prop} Assume statement \ref{Holderboundpluripotentialtheory}.
Statement \ref{localreductiontoquadric2} holds for $n=3$.	Morever, for any $0<\alpha/2<\frac{1}{3}$, there is a uniform H\"older estimate
\begin{equation}
d_{ \omega'  }(x, x') \leq C d_{ \omega_{Fubini}|_{X_y   }   } (x, x')^{\alpha/2}
\end{equation}
for any two points $x, x'$ on $X_y$ and the constant is independent of $y$ with $|y|\leq 1$.
\end{prop}

\begin{proof}
We first make a base change to the square root fibration:
\[
\begin{tikzcd}
\{  Z_1^2+Z_2^2+Z_3^2=w^2 Z_0^2    \} \arrow{r}{} \arrow[swap]{d}{\tilde{f} } & \{ Z_1^2+Z_2^2+Z_3^2=yZ_0^2  \}\subset \C_y \times \mathbb{P}^3 \arrow{d}{f} \\
\C_w  \arrow{r}{y=w^2} & \C_y
\end{tikzcd}
\]
In reality we only consider small values of $y$, say $|y|\leq 1$.
The fibre product has a three-fold ODP singularity at the point $w=0, Z_1=Z_2=Z_3=0$. This admits a small resolution in a completely standard way:
\[
\tilde{\mathcal{X}} \to \{  Z_1^2+Z_2^2+Z_3^2=w^2 Z_0^2    \}.
\]
In the local piece $Z_0=1$, this family $\tilde{\mathcal{X}}  $ can be described explicity as 
\[
\mathcal{O}_{\mathbb{P} ^{1} }  (-1) \oplus \mathcal{O}_{\mathbb{P}^1 }  (-1) =\{ [U:V], (u_1, v_1)\otimes (u_2, v_2) |  Uv_1=Vu_1, (U,V)\neq (0,0)      \}.
\]
Here $u_1, u_2, v_1, v_2$ are not well defined by themselves, but $u_1u_2, u_1v_2, u_2v_1, u_2v_2$ are.
The resolution map is 
\[
Z_1+\sqrt{-1} Z_2=u_1u_2, w+Z_3=v_1u_2, w-Z_3=v_2 u_1, Z_1-\sqrt{-1} Z_2= v_1v_2.
\]

The point is that the family $\tilde{\mathcal{X}}\to \C_w$ has the same fibres as the square root fibration $\tilde{f}$, for $w\neq 0$, and makes the fibration a submersion. This means for $|y|\leq 1$ the smooth fibres of $f$ can be endowed with K\"ahler metrics $\omega_{ \tilde{\mathcal{X}}_w  }$ which have uniformly bounded geometry, for example using the restriction of some Fubini-Study metric on the quasiprojective variety $\tilde{\mathcal{X}}$. We shall from now on make identification of the fibres under the natural maps.

  Now we focus on $\{Z_0=1, |Z_1|^2+|Z_2|^2+|Z_3|^2 \leq 1\}$. It can be shown by simple explicit calculation that in this region
\begin{itemize}
\item $H$ is uniformly equivalent to $(|u_1|^2+|v_1|^2)(|u_2|^2+|v_2|^2)$.
\item The local potential of $\omega_{Fubini}|_{X_y}$ is smoothly bounded on the fibre $\tilde{\mathcal{X}}|_w $ with respect to the norms defined by $\omega_{\tilde{\mathcal{X}}_w } $.
\item  If $U=1, |V|\leq 1$, then the local coordinates on $ \tilde{\mathcal{X}}|_w     $ are given by $V$ and $u_1u_2$. We have \[\omega_{Fubini}|_{X_y} \leq C(\sqrt{-1}(d(u_1u_2) \wedge d\overline{u_1u_2}  + H \sqrt{-1} dV\wedge d\bar{V}  ) \leq C \omega_{ \tilde{\mathcal{X}}_w  }.\] If $V=1, |U|\leq 1$, then the local coordinates on $ \tilde{\mathcal{X}}|_w     $ are given by $U$ and $v_1v_2$. We have \[\omega_{Fubini}|_{X_y} \leq C( \sqrt{-1}d(v_1v_2) \wedge d\overline{v_1v_2} +   H \sqrt{-1} dU\wedge d\bar{U}  )\leq C \omega_{ \tilde{\mathcal{X}}_w  } .\]
\item The volume form $\omega'^2 \leq \frac{C}{H} \omega_{Fubini}^2|_{X_y} \leq C \omega_{ \tilde{\mathcal{X}}_w     }^2$. This can be seen more conceptually by noticing the small resolution restricted to the central fibre is crepant, and thus the holomorphic volume form on the local part of $X_{y=0}$ and on $\tilde{\mathcal{X}}_{w=0} $ can only differ by a nowhere vanishing holomorphic function.
\end{itemize}
All constants are of course meant to be independent of $y$ for $|y|\leq 1$.

Let us understand how the class of $[\omega_{Fubini}|_{X_y}]= [\omega']$ in Statement \ref{localreductiontoquadric2} compares with $[\omega_{ \tilde{\mathcal{X}}_w    }]$. Here the subtlety is that the smooth reference metric $\omega_{ \tilde{\mathcal{X}}_w    }$ cannot have the same K\"ahler class as $[\omega']$. In particular on the central fibre $\tilde{\mathcal{X}}_0$, the metric $\omega_{ \tilde{\mathcal{X}}_0    }$ is smooth and positive, while $\omega_{Fubini}|_{X_0}$ is a semi-positive smooth form pulled back from the orbifold $\{Z_1^2+Z_2^2+Z_3^2=0 \} $, whose class lies on the boundary of the K\"ahler cone. This subtlety means the uniform H\"older bound on the potential, Theorem A* in \cite{Ko}, no longer applies, and one is forced to envoke a generalised conjectural statement \ref{Holderboundpluripotentialtheory}.

Now given the volume estimate
$
\omega'^2 \leq C\omega_{ \tilde{\mathcal{X}}_w     }^2,
$
by the method in section \ref{Generalboundsondistancefunctions} (\cf also the remarks there), we have now a uniform H\"older bound on the distance function of $\omega'$ for all smooth fibres, with respect to the background metric $\omega_{\tilde{\mathcal{X}}_w   }$. Here $L^p=L^\infty$, and $q=1$. The H\"older exponent $\alpha$ can be taken to be any positive number smaller than  $\frac{2}{3}$. This means for any $\alpha/2 < \frac{1}{3}$, there is some uniform constant, for which 
\[
d_{ \omega'  }(x, x') \leq C d_{ \omega_{\tilde{\mathcal{X}}_w   }   } (x, x')^{\alpha/2}.
\]
In particular this proves the uniform diameter bound. 

We can be a little more accurate on the H\"older exponent, although this shall not be required in the rest of this paper. Geometrically, each smooth quadric is isomorphic to $\mathbb{P}^1 \times \mathbb{P}^1$, and the level sets of $\frac{U}{V}$ are just lines in a preferred ruling. But along such lines, the metric $\omega_{Fubini}|_{X_y}$ is uniformly equivalent to $\omega_{\tilde{\mathcal{X}}_w   }$, so 
\[
d_{ \omega'  }(x, x') \leq C d_{ \omega_{Fubini}|_{X_y   }   } (x, x')^{\alpha/2}
\] if $x, x'$ are on the same line. Now we observe that we could equally well have chosen the other small resolution and work with a different ruling, which gives the same estimate. Since we can join any two points by two lines, one from each ruling, we see that this H\"older estimate holds for any two points on $X_y$.
\end{proof}

\begin{thm} Assume Statement \ref{Holderboundpluripotentialtheory}.
The conjecture \ref{uniformfibrediameterboundconjecture} holds for the special case $n=3$, in particular for a K3-fibred Calabi-Yau 3-fold with at worst nodal fibres. Morever for every exponent
$\alpha/2<\frac{1}{3}$,
there is a uniform H\"older bound
\[
d_{ \frac{1}{t} \tilde{\omega}_t|_{X_y}  }(x, x') \leq C d_{ \omega_X|_{X_y   }   } (x, x')^{\alpha/2}
\]
for any $x, x' \in X_y$, where the estimate is uniform in all fibres.
\end{thm}

\section{Limit of fibrewise metrics}

\subsection{Convergence behaviour over one dimensional base}\label{convergenceoveronedimensionalbasesection}

Let $t$ be small throughout the discussion. In the presence of smooth bounds (cor \ref{smoothboundfibrewise}), as long as we can show a convergence bound on the normalised potential on the fibre, we have a smooth convergence result for fibre metrics. This section is a step in this direction. To avoid excessive complication of conditions, the reader is encouraged to primarily think of a Lefschetz fibration with $n\geq 3$, and assume the uniform diameter bound (\ref{uniformfibrediameter}). The crucial feature we need is the one dimensional base condition. In fact, if we stay away from the singular fibres, then no assumption on the nature of singularity is needed for our arguments.




We fix a fibre $X_y$. Consider a coordinate ball $B(y, R)$ on the base $Y$.
Let $\omega_{Y,y}$ be the Euclidean metric on $B(y,R)$, and at the centre $y\in Y$, we have the normalisation $\omega_{Y,y}|_y=f_*(i^{n^2}\Omega\wedge \overline{ \Omega})|_y$. This is controlled by the base metric $\omega_Y$ by condition (\ref{volumecondition}).

\begin{prop}\label{concentrationestimate}
Assume the base is one dimensional, and the volume condition (\ref{volumecondition}) and the uniform fibre diameter bound (\ref{uniformfibrediameter}) hold, then we have a concentration estimate for $\Tr_{\tilde{\omega_t}} \omega_{Y,y}$, which holds uniformly for all $y\in Y$:
\begin{equation}
\max_{|f-y|\leq t^{1/2}} \Tr_{\tilde{\omega_t}}\omega_{Y,y}
\leq \frac{n}{a_t}+C/|\log t|,
\end{equation}
\begin{equation}
t^{-n}\norm{    \Tr_{\tilde{\omega_t}}\omega_{Y,y} -\frac{n}{a_t}       } _{L^1_{\tilde{\omega_t}}(|f-y|\leq t^{1/2})}  \leq \frac{C}{|\log t|}.
\end{equation}
\end{prop}

\begin{proof}
By the Chern-Lu formula, using the fact that $\omega_{Y,y}$ has zero curvature, we have
\[
\Lap_{\tilde{\omega_t}} (\log \Tr_{\tilde{\omega_t} }\omega_{Y,y}  ) \geq  0.
\]
The rest of the argument is essentially an effective version of the fact that $\C$ cannot support bounded subharmonic functions.

Notice that $\log |f-y| $ is a pluriharmonic function on $X\setminus X_y$, so in particular is harmonic with respect to the Laplacian $\Lap_{\tilde{\omega_t}}$. In particular we can compare the function $\log \Tr_{\tilde{\omega_t} }\omega_{Y,y} $ with $C \log |f-y|+C'$ using the maximum principle to achieve the following three circle type inequality:
\begin{equation*}
\begin{split}
\max_{|f-y|=D t^{1/2}} \log \Tr_{\tilde{\omega_t} }\omega_{Y,y} -\max_{|f-y|=t^{1/2}} \log \Tr_{\tilde{\omega_t} }\omega_{Y,y} \\
\leq \frac{\log D}{\log R/t^{1/2}}(\max_{|f-y|=R} \log \Tr_{\tilde{\omega_t} }\omega_{Y,y} -\max_{|f-y|=t^{1/2}} \log \Tr_{\tilde{\omega_t} }\omega_{Y,y})
\end{split}
\end{equation*}
where $D$ is a large constant to be chosen later.
Now notice by the bound in proposition \ref{omega0isbounded}, we have $\log  \Tr_{\tilde{\omega_t} }\omega_{Y,y}\leq C$. Morever, since by the volume condition (\ref{volumecondition}), on any fibre near $X_y$ there is some region where $H\geq C$ and therefore $\omega_t$ is uniformly equivalent to $\tilde{\omega_t}$, we see
$\max_{|f-y|=t^{1/2}} \log \Tr_{\tilde{\omega_t} }\omega_{Y,y} \geq C$ as well. Therefore we obtain the following inequality, which says on the region $\{|f-y|\leq Dt^{1/2}\}$ the maximum is almost achieved in the interior for the subharmonic function $\log \Tr_{\tilde{\omega_t} }\omega_{Y,y}$:
\begin{equation}
\max_{|f-y|\leq Dt^{1/2}} \log \Tr_{\tilde{\omega_t} }\omega_{Y,y} -\max_{|f-y|\leq t^{1/2}} \log \Tr_{\tilde{\omega_t} }\omega_{Y,y} \leq \frac{C}{|\log t|}.
\end{equation}
We want to use this to conclude the function $\log \Tr_{\tilde{\omega_t} }\omega_{Y,y}$ is close to a constant.

Now we observe that the region $\{|f-y|
\leq D t^{1/2}\}$ with the metric $\tilde{\omega_t}$ satisfies the local volume noncollapsing estimate (\ref{localvolumenoncollapsing2}). Since the metric $\tilde{\omega_t} $ is Ricci flat, it satisfies the scale invariant Sobolev inequality and the Poincar\'e inequality on the region $\{|f-y|\leq Dt^{1/2}\}$ (\cf \cite{Hein}). Morever, using the bounds on the metric $\tilde{\omega_t}$ we can arrange the constant $D$ to be large enough, such that there is a point $P\in X_y$ and a radius $r$ with \[
\{|f-y|\leq t^{1/2} \} \subset B_{\tilde{\omega_t}}(P,r)\subset B_{\tilde{\omega_t}}(P,4r) \subset \{|f-y|\leq Dt^{1/2} \}
\]
We have the following Harnack inequality (see \cite{HanLin}, Page 83-89 for a proof which only relies on Sobolev and Poincar\'e inequalities):
\begin{lem}
If $v$ is a nonnegative superharmonic function in $B(P,4r)$, then for $1\leq p<2n/(2n-2)$, we have a constant independent of $r, v$, such that
\begin{equation}
r^{-2n/p}\norm{v}_{L^p(B_{\tilde{\omega}_t}(P,2r))}\leq C\inf_{B(P,r)}v.
\end{equation}
\end{lem}
Apply this to the function $v=-\log \Tr_{\tilde{\omega}_t}\omega_{Y,y} +\max_{|f-y|\leq Dt^{1/2}} \log \Tr_{\tilde{\omega}_t }\omega_{Y,y} $, we get
\[
r^{-2n}\norm{v}_{L^1(B_{\tilde{\omega}_t}(P,2r))}\leq \frac{C}{|\log t|}.
\]
In particular
\[
t^{-n}\norm{v}_{L^1_{\tilde{\omega}_t}(|f-y|\leq t^{1/2})} \leq \frac{C}{|\log t|}.
\]
The above estimate can be thought of as an effective version of the strong maximum principle.

Since $v\geq 0$, we have $v\geq C(1-e^{-v})$, and using also proposition \ref{omega0isbounded}, we deduce there is an $L^1$ bound
\[
t^{-n}\norm{   - \Tr_{\tilde{\omega}_t}\omega_{Y,y} +\max_{|f-y|\leq t^{1/2}} \Tr_{\tilde{\omega}_t}\omega_{Y,y}       } _{L^1_{\tilde{\omega}_t}(|f-y|\leq t^{1/2})}  \leq \frac{C}{|\log t|}.
\]
This says the deviation of the maximum from the average is small of order $O(\frac{1}{|\log t|})$.
But the average value of $\Tr_{\tilde{\omega}_t}\omega_{Y,y}$ is known:
\[
f_*((\Tr_{\tilde{\omega}_t}\omega_{Y,y}) \tilde{\omega}_t^{n})=n f_*(\tilde{\omega}_t^{n-1}\wedge \omega_{Y,y})=nt^{n-1}f_*(\omega_X^{n-1})\wedge \omega_{Y,y}=nt^{n-1}\omega_{Y,y},
\]
So \[
\int_{|f-y|\leq t^{1/2} }\Tr_{\tilde{\omega}_t}\omega_{Y,y}\tilde{\omega}_t^{n}=n t^{n-1} Vol_{\omega_{Y,y}}(B(y,t^{1/2})).
\]
Observe also 
\[
\int_{|f-y|\leq t^{1/2}} \tilde{\omega}_t^{n}=a_t t^{n-1} \int_{|f-y|\leq t^{1/2}} i^{n^2}\Omega\wedge \overline{\Omega}=a_t t^{n-1} Vol_{\omega_{Y,y}}(B(y,t^{1/2})),
\]
where the $a_t$ is the volume normalising constant in the Monge-Amp\`ere equation, and we used the particular choice of normalisation for $\omega_{Y,y}$. Therefore the average value of $\Tr_{\tilde{\omega}_t}\omega_{Y,y}$ is $\frac{n}{a_t}$, and the estimate follows.
\end{proof}

In the following computation we need to make use of $\omega_{SF,y}$ on the singular fibres. To ensure this is well defined, we shall assume the singular fibres are Calabi-Yau varieties. (\cf \cite{EGZ}).

Now we observe that the fibrewise volume form can be expressed in terms of $\Tr_{\tilde{\omega}_t}\omega_{Y,y}$, by the following computation:
\[
\frac{\tilde{\omega}_t|_{X_y}^{n-1}} { \omega_{SF,y}^{n-1 }   }
=\frac{\tilde{\omega}_t^{n-1}\omega_{Y,y}   } { \omega_{SF,y}^{n-1 }\omega_{Y,y}   }
=\frac{\Tr_ { \tilde{\omega}_t   }{ \omega_{Y,y}  }     \tilde{\omega}_t^{n}  } {n \omega_{SF,y}^{n-1 }\omega_{Y,y}   }
=\frac{a_t t^{n-1}\Tr_ { \tilde{\omega}_t   }{ \omega_{Y,y}  }     i^{n^2} \Omega \wedge \overline{\Omega}  } {n \omega_{SF,y}^{n-1 }\omega_{Y,y}   },
\]
and using the adjunction formula description of $\omega_{SF,y}$, the quantity $\frac{     i^{n^2} \Omega \wedge \overline{\Omega}  } { \omega_{SF,y}^{n-1 }\omega_{Y,y}   }$ is in fact constant on the fibre, so equals the constant $\frac{  f_*( i^{n^2} \Omega \wedge \overline{\Omega}   )  }{  \omega_{Y,y}  }=1
$ at the point $y$. Hence we have the formula for the fibrewise volume:
\begin{equation}
\frac{\tilde{\omega}_t|_{X_y}^{n-1}} { \omega_{SF,y}^{n-1 }   }=\frac{a_t t^{n-1}\Tr_ { \tilde{\omega}_t   }{ \omega_{Y,y}  }     } {n }
\end{equation}
Comparing this with the concentration estimate (proposition \ref{concentrationestimate}), we see that the normalized fibre volume form $(\frac{\tilde{\omega}_t|_{X_y} }{t})^{n-1} $ with respect to volume form of the background metric $\omega_{SF,y}$ is almost bounded above by one, and satisfies an $L^1$ concentration estimate. Now recall we introduced a function $v$ in our proof of the fibrewise oscillation estimate (\ref{fibrepotentialoscillation}), defined by $\frac{1}{t}\tilde{\omega}_t|_{X_y}=\omega_{SF,y}+dd^c v$. The standard manipulation for the Monge-Amp\`ere equation using integration by part yields
\[
\int_{X_y} v( \omega_{SF,y}^{n-1 } - \ (\frac{1}{t}\tilde{\omega}_t|_{X_y})^{n-1}         ) \geq C\int_{X_y} |\nabla v|^2_{\omega_{SF,y} } {\omega_{SF,y} }^{n-1}.
\]
But the oscillation of $v$ is bounded in $L^{\infty}$, so in fact we obtain the following
\begin{prop}Assume the base is one dimensional, the volume condition (\ref{volumecondition}) and the uniform fibre diameter bound (\ref{uniformfibrediameter}) hold, and all singular fibres are Calabi-Yau varieties.
The fibrewise potential has the convergence estimate
\begin{equation}
\int_{X_y} |\nabla v|^2_{\omega_{SF,y} } {\omega_{SF,y} }^{n-1} \leq \frac{C}{|\log t|}.
\end{equation}
\end{prop}

\begin{rmk}
The potential $v$ has some ambiguity in its normalisation. One of the possibilities is $\int_{X_y} v\omega_{SF,y}^{n-1}=0$. All other reasonable choices are equivalent for the purpose of estimation, and do not change the metric anyway.
\end{rmk}

\begin{thm}
In the setup above, away from the critical points, \ie in the region $H\geq C$, the normalised fibrewise metric $\frac{1}{t}\tilde{\omega}_t|_{X_y}$ converges smoothly to $\omega_{SF,y}$ with constructive estimates to all orders: for any small $\epsilon>0$,
\begin{equation}\label{fibrewisesmoothconvergence}
\norm{\nabla^{(k)}_{\omega_X} (\omega_{SF,y}-\frac{1}{t}   \tilde{\omega}_t|_{X_y} )  }_{ {X_y}\cap \{H\geq C \}} \leq \frac{C(k, \epsilon)}{ |\log t|^{1/2-\epsilon}  }.
\end{equation}
\end{thm}

\begin{proof}
Away from the critical points $\omega_{SF,y}$ is smoothly equivalent to $\omega_X$. Now given the convergence estimate $\norm{\nabla v}_{L^2}^2 =O(\frac{1}{|\log t|})$, and the smooth bound in corollary \ref{smoothboundfibrewise}, it is a relatively standard fact that the higher order derivatives of $v$ are small to any given order. The precise effective version is the above inequality (\ref{fibrewisesmoothconvergence}).
Notice the argument works even on part of the singular fibre.
\end{proof}

Due to its importance, we separately state

\begin{thm}
If the base is one dimensional, without conditions on types of singularities, then away from the critical locus $S$, the normalised fibrewise metric $\frac{1}{t}\tilde{\omega}_t|_{X_y}$ converges smoothly to $\omega_{SF,y}$ with constructive estimates to all orders.
\end{thm}

\begin{rmk}
V. Tosatti informs the author that
the convergence result away from the singular fibre was already known (see \cite{TosattiWeinkoveYang}, \cite{TosattiZhang}) without effective estimate.
\end{rmk}

\begin{rmk}
The convergence rate is logarithmic according to our argument, but one may expect that in reality the rate is much faster away from singular fibres, if one is willing to believe the K\"ahler potential admits a formal power series expansion in $t$.
\end{rmk}

\begin{rmk}
There is a subtlety about the meaning of smooth convergence. It is precisely given by (\ref{fibrewisesmoothconvergence}). The limit $\omega_{SF,y}$ does not need to depend smoothly on the parameter $y$, since even the integral $F_y=\int_{X_y} i^{(n-1)^2}\Omega\wedge \overline{\Omega}$ does not depend smoothly on $y$ at the critical value.
\end{rmk}

\begin{rmk}\label{RongandZhangsmoothconvergence}
A related result in \cite{Zhang}, theorem 1.4, which we shall use later, is concerned with a smoothing of a Calabi-Yau variety, \ie a flat projective family of Calabi-Yau varieties over a small disc, with trivial relative canonical bundle and smooth generic fibre. They consider an ample line bundle on the total space, which defines a unique Calabi-Yau metric on each smooth fibre in the first Chern class, and show that around any compact subset of the central fibre away from the singular set, one can find trivialisations such that the Calabi-Yau metrics on nearby fibres converge smoothly to the singular Calabi-Yau metric on the central fibre.

The author thinks their integral K\"ahler class assumption is not essential, and their proof still works if the K\"ahler classes of the Calabi-Yau metrics on the fibres are defined by the restriction of a global K\"ahler class on a smooth total space. Yuguang Zhang confirms this claim in private communication.
\end{rmk}

\subsection{Gradient estimate}

This is a continuation of the last section. The aim is to improve certain aspects of the concentration estimate (proposition \ref{concentrationestimate}), and show that $df$ is in some sense approximately parallel.

\begin{prop}
In the setup of last section, we have the gradient estimate
\begin{equation}\label{gradientestimateChernLu}
\int_{ |f-y|\leq  t^{1/2}    }(\frac{|\nabla df|^2 }{\Tr_{ \tilde{\omega}_t  } \omega_{Y,y}  }-|\partial \log \Tr_{ \tilde{\omega}_t  } \omega_{Y,y} |^2) \tilde{\omega}_t^n \leq \frac{C t^{n-1}} {|\log t|}.
\end{equation}
\end{prop}

\begin{proof}
The starting point is that the Chern-Lu formula has a more refined version (\cf \cite{ChernLu}, proposition 7.1)
\[
\Lap_{\tilde{\omega}_t } \log \Tr_{ \tilde{\omega}_t  } \omega_{Y,y} =\frac{|\nabla df|^2 }{ |d f|^2 }-|\partial \log |df|^2 |^2.
\]
Here the RHS is pointwise non-negative by Cauchy-Schwarz. The norm square of the differential $df$ is just the trace: $|df|^2=\Tr_{ \tilde{\omega}_t  } \omega_{Y,y}$.
This is consistent with the interpretation of $df$ as a section of $\Hom (TX, TY)$, and the natural norms come from the metric $ \tilde{\omega}_t  $ and $\omega_{Y,y}$.

Now let $\chi$ be a cutoff function on the base $Y$, with support in $B_{\omega_{Y,y}}(y, t^{1/2})$, and equals to one on $B_{\omega_{Y,y}}(y,\frac{1}{2} t^{1/2})$. We require \[-\frac{C}{t}\omega_{Y,y}\leq i\partial \bar{\partial} \chi \leq \frac{C}{t}\omega_{Y,y},\]
so the Laplacian $|\Lap_{\tilde{\omega}_t } \chi |\leq \frac{C}{t}$. We now multiply the test function $\chi$ to both sides of the refined Chern-Lu formula, to get
\[
\int_{ |f-y|\leq t^{1/2}    }\chi ( \frac{|\nabla df|^2 }{ |d f|^2 }-|\partial \log |df|^2 |^2) =\int \Lap_{\tilde{\omega}_t } \chi (\log \Tr_{ \tilde{\omega}_t  }  \omega_{Y,y}-\max_{ |f-y|\leq t^{1/2}   }{\log \Tr_{ \tilde{\omega}_t  }  \omega_{Y,y}   }  )         .
\]
Using the proof of the concentration estimate, the $L^1$ of the $\log$ term is controlled, so the RHS is estimated by $\frac{C}{t}\frac{t^n}{|\log t|}$. This implies 
\[
\int_{ |f-y|\leq \frac{1}{2} t^{1/2}    }(\frac{|\nabla df|^2 }{ |d f|^2 }-|\partial \log |df|^2 |^2)  \tilde{\omega}_t^n \leq \frac{C t^{n-1}} {|\log t|}.
\]
This is the claim up to cosmetic changes of constants.
\end{proof}

\subsection{Gromov Hausdorff limit around the singular fibre}\label{GHlimitaroundsingularfibresection}

In this section let us for simplicity concentrate on the case of Lefschetz fibrations, with the uniform diameter bound (\ref{uniformfibrediameter}). The actual arguments are more general in character. The idea of using an embedding map to compare complex structures is inspired by the work of Donaldson and Sun \cite{DonSun}.

 We wish to consider the convergence of the metric restricted to the singular fibre. The expected limit, after rescaling, is the nodal Calabi-Yau metric on the fibre $\omega_{SF,y}$. Notice we certainly cannot expect $\frac{1}{t}\tilde{\omega}_t|_{X_y}$ to have a uniform two sided $C^2$ bound by $\omega_{SF,y}  $, because this is not true even for $t=1$. This suggests we should study a weaker notion of convergence, such as the Gromov-Hausdorff limit. This is a sensible question in the light of the local volume noncollapsing estimate (\ref{localvolumenoncollapsing1}), (\ref{localvolumenoncollapsing2}).

We pick a point $P$ on the fibre of interest, which can be singular, and look at the pointed sequence of Ricci flat spaces $Z_t=(f^{-1} B_{\omega_Y} (f(P), R)\subset X, \frac{1}{t}\tilde{\omega}_t)$. The local noncollapsing implies that after passing to subsequence, there is some Gromov-Hausdorff limit space $(Z,\omega_\infty)$, which is complex $n$-dimensional, with a regular locus $Z^\text{reg}$, which is connected, open, dense, where the limiting metric is smooth. The singular locus has Hausdorff codimension at least 4. Morever, the smooth locus $Z^\text{reg}$ has a natural limiting complex structure, such that the limiting metric is K\"ahler. (See \cite{DonSun} for a nice summary.) In the arguments below, we shall suppress mention of subsequence to avoid overloading notation, and tacitly understand a Gromov-Hausdorff metric is fixed on the disjoint union $Z_t \sqcup Z   $, which displays the GH convergence. We will also assume implicitly that $t$ is sufficiently small.

Now we wish to identify the complex structure. Let us first mention some heuristic: since everything away from the fibre $X_{f(P)}$ is pushed to infinity by scaling, the limit as a complex variety should be the normal neighbourhood of $X_{f(P)}$, which is just the trivial product $X_{f(P)}\times \C$ in the case of a smooth fibre, and the guess is that the same is true for the singular fibre.

More formally, we need to build comparison maps. Let $u$ denote the standard coordinate on $\C$, and $\omega_\C$ refers to the standard Euclidean metric on $\C$. Define the maps \[f_t: (Z_t, \frac{1}{t}\tilde{\omega}_t  ) \to (X\times \C, \omega_X+\omega_{\C} ) \] via $x\mapsto (x, t^{-1/2}(f-f(P)))$ (Compare with our section \ref{smoothcontrolawayfromcriticalpoints} on smooth control away from singularity.) These maps are holomorphic, and by the uniform bound $\Tr_{\tilde{\omega}_t  } {\omega_t} \leq C $, the Lipschitz constants are bounded above independent of $t$. This means the Gromov-Hausdorff limit inherits a Lipschitz map $f_\infty$ into $X\times \C$, which is a subsequential limit of $f_t$. By the interior regularity of holomorphic functions, the limiting map $f_\infty$ is holomorphic. We identify the image:

\begin{lem}
The image of $f_\infty$ is $X_{f(P)}\times \C$.
\end{lem}

\begin{proof}
First we show the image is contained in $X_{f(P)}\times \C$. For any fixed radius $D$, the preimage of $B_{Euclid}(0,D)\subset \C$ under $f_t$ is contained in a neighbourhood of $X_{f(P)}\subset (Z_t, \frac{1}{t}\tilde{\omega}_t  )$, with distance scale $CDt^{1/2}$ in the $\omega_X$ metric. So we have
\[
\text{Image}(f_t) \cap pr_{\C}^{-1}( B(0,D)  )   \subset \{ (x, u) \in X\times \C : dist_{\omega_X}(x, X_{f(P)}) \leq CDt^{1/2} , |u|\leq D    \}.
\]
By continuity, 
\[
\text{Image}(f_\infty) \cap pr_{\C}^{-1}( B(0,D)  )   \subset \{ (x, u) \in X\times \C : dist_{\omega_X}(x, X_{f(P)}) = 0 , |u|\leq D    \},
\]
so the first claim is proved.

Next we show the image contains $X_{f(P)}\times \C$. This follows from the fact that for any $(x, u) \in X_{f(P)}\times \C $ with $ |u|\leq D $, suppose $t$ is sufficiently small depending on $D$, there is always a point $Q_t\in Z_t$, whose  distance to $P$ is bounded by $C(1+D)$ in $\frac{1}{t}\tilde{\omega}_t$ metric, and $d(f_t(Q), (x,u) )\leq CDt^{1/2}$ in the metric on the target. This fact comes from the fibre diameter bound and our discussions around volume noncollapsing. Then using the compactness of the ball $B(P, C(1+D)) \subset Z$, we can extract a subsequential limit $Q$, which has to map to $(x,u)$ by continuity.
\end{proof}

Our next goal is to establish good properties of $f_\infty$ over the smooth locus of $X_{f(P)}\times \C$.

We recall that the function $H$ measures the severity of singular effect. Now $H$ is a continuous function of $X_{f(P)}$, so defines a function on $Z$ by pulling back via $f_\infty$. For any fixed constant $h>0$, if a point $Q$ on $Z$ satisfies $H> h$, then by continuity, it is close in the Gromov Hausdorff metric to a point $Q_t$ on $Z_t$, which has $H>h$, where $t$ can be taken sufficiently small. Recall (\cf proposition \ref{smoothboundawayfromsingularity}) that in appropriate coordinates, which can be taken here as $z_1, \ldots z_{n-1}, u-pr_\C f_{\infty}(Q)$ where $u=\frac{1}{t^{1/2}} (f-f(P))$, the ball $B_{\frac{1}{t}\tilde{\omega}_t }(Q_t, C)\subset \{H>h\}$ can be regarded as a bounded domain in $\C^n$, containing (up to a bounded factor) the unit ball in $\C^n$, and the metric is smoothly equivalent to the Euclidean metric, with smooth bounds. So by shrinking the radius slightly, $B_{\frac{1}{t}\tilde{\omega}_t }(Q_t, C)$ subconverges to a smooth metric, which has to agree with $B(Q,C) \subset Z$, so $Q\in Z^\text{reg}$. 

\begin{rmk}
These good holomorphic coordinates $z_1, \ldots, z_{n-1},u$ are also tautologically local functions on $X\times \C$. In more formal language, the embedding map \[B_{\frac{1}{t}\tilde{\omega}_t }(Q_t, C) \xrightarrow{(z_1, \ldots, u)} \C^n\] factors through the map $f_t$ to a fixed open set $U$ in $X\times \C$  independent of the small $t$ :
\[
B_{\frac{1}{t}\tilde{\omega}_t }(Q_t, C) \xrightarrow{f_t} U \xrightarrow{(z_1, \ldots, u)} \C^n
.\]	
\end{rmk}
We can take the limit to obtain local functions on $Z$:
\[
B_{\omega_\infty}(Q, C) \xrightarrow{f_\infty} U \xrightarrow{(z_1, \ldots, u)} \C^n
.
\]
We can morever assume that the radius $C$ is chosen so that all those maps above from the balls into $\C^n$ are uniformly bi-Lipschitz, so a holomorphic inverse can be defined on the image. This implies $f_\infty$ restricted to $B_{\omega_\infty}(Q, C)$ is a complex isomorphism, and $z_1, \ldots,u$ can be regarded as local coordinates on $Z$, around the point $Q$.

\begin{lem}
The map $f_\infty$ is injective over the smooth part of $X_{f(P)}\times \C$.
\end{lem}

\begin{proof}
We assume $Q$ and $Q'$ have the same image in $X_{f(P)}\times \C$, but $d(Q,Q')>0$.  By the local isomorphism result above, $Q$ and $Q'$ cannot be in the same chart as above, so their distance is bounded below by some constant depending on $Q$: $d(Q,Q')>C.$ Then we can find $Q_t$, $Q_t'$ on $Z_t$ close to $Q,Q'$ in GH metric, so $d(Q_t,Q_t')>C$. Using the fact that $f_t$ is an embedding map, and near the point $Q_t$ it is bi-Lipshitz, we see that $d(f_t(Q_t), f_t(Q_t')) \geq C$. Take the limit to obtain $d(f_\infty(Q), f_\infty(Q')) \geq C$, contradiction.
\end{proof}

\begin{prop}
The map $f_\infty$ is a biholomorphism  $\{H>0\} \subset Z\to X_{f(P)}^{\text{reg} }\times \C$.
\end{prop}

\begin{proof}
We have seen injectivity and local complex isomorphism. The map $f_\infty: Z\to X_{f(P)}\times \C$ is surjective, and the preimage of $ X_{f(P)}^{\text{reg} }\times \C$ is precisely $\{H>0\} \subset Z$.
\end{proof}

We observe next that the sequence of Calabi-Yau metrics carry holomorphic volume forms, and therefore canonically defined volume measures. In local coordinates $z_1,\ldots, z_{n-1}, y$ (recall $y$ is the coordinate on $Y$), we can write $\Omega=h dy\wedge dz_1\wedge \ldots dz_{n-1}$, for some holomorphic function $h=h(z_1, z_2,\ldots, y)$. Since $\tilde{\omega}_t^n=a_t t^{n-1} i^{n^2} \Omega \wedge \overline{\Omega}$, we have the Monge-Amp\`ere equation in coordinates:
\[
(\frac{1}{t} \tilde{\omega}_t )^n =a_t   |h|^2 \prod_{1}^{n-1}{ \sqrt{-1} dz_i d\bar{z_i}  }
\wedge \sqrt{-1}du d\bar{u}.
\]
Since we are restricting attention to the region with $H>C$, in the good holomorphic coordinate $z_1, \ldots, u$, the metrics $\tilde{\omega}_t$ subconverge smoothly to $\omega_\infty$, and the limit satisfies the Monge-Amp\`ere equation on the chart:
\[
\omega_{\infty}^n =a_0   |h|^2 \prod_{1}^{n-1}{ \sqrt{-1} dz_i d\bar{z_i}  }
\wedge \sqrt{-1}du d\bar{u}.
\] 
Here $h=h(z_1,z_2,\ldots, z_{n-1},f(P))$ since $y=f(P)+ut^{1/2}$ tends to the constant $f(P)$ as $t\to 0$ in a chart with bounded $u$. We recall $\Omega=dy\wedge \Omega_y$, so by comparison $\Omega_{f(P)}=hdz_1\dots dz_{n-1}|_{X_{f(P)}}$, and 
$\omega_\infty^n=a_0 i^{(n-1)^2}\Omega_{f(P)}\wedge \overline{\Omega_{f(P)}} \wedge { \sqrt{-1} du \bar{du}  }$. Integrate over the region $\{ H>0 \}\cap \{ |u|\leq D  \} \subset Z$, for $D$ any fixed number, we find
\[
\begin{split}
\int_{ \{ H>0 \}\cap \{ |u|\leq D  \} } \omega_\infty^n=a_0 \int_{ X_{f(P)} } i^{(n-1)^2}\Omega_{f(P)}\wedge \overline{\Omega_{f(P)}}  \int_{\{|u|\leq D \}\subset \C  } { \sqrt{-1} du d\bar{u}  }   \\
= \lim \int_{\{|u|\leq D \} \subset Z_t  }
a_t i^{n^2} t^{-1} \Omega\wedge \overline{\Omega}  
=\lim \int_{\{|u|\leq D \} \subset Z_t  } (\frac{1}{t} \tilde{\omega}_t )^n.
\end{split}
\]
Here the intermediate equality uses the continuity of $F_y=\int_{X_y} i^{(n-1)^2}\Omega_y\wedge\overline{\Omega_y}$, which is checked for Lefschetz fibrations with $n\geq 3$.

But in a non-collapsing situation, with non-negative Ricci curvature, we have the measured Gromov-Hausdorff convergence (\cf Cheeger's book \cite{Cheeger}, Page 66), which implies
\[
\liminf Vol_{ \frac{1}{t} \tilde{\omega}_t    }( \{ |u|\leq D  \} \subset Z_t  ) \geq Vol_{ \omega_\infty }( \{|u|\leq D \}\subset Z   ).
\]
Comparing the above, we obtain:
\begin{prop} (Full measure property)
The set $H>0$ inside $Z$ must have full measure on each cylinder $\{ |u|\leq D  \} \subset Z$, so the set $H=0$ has measure zero in $Z$. In particular, every ball in $Z$ contains a point with $H>0$.
\end{prop}

We now study the metric $\omega_\infty$ over the smooth region $X_{f(P)}^{\text{reg} }\times \C$.

\begin{prop}
Over the smooth part of $X_{f(P)}\times \C$, the limiting metric restricts fibrewise to the Calabi-Yau metric $\omega_{SF,f(P)}$.
\end{prop}

\begin{proof}
Let $Q$ be a point with $H> 0$ on $Z$. Then we look at the metric in the local coordinates $z_1,\ldots, z_{n-1}, u$. The same set of functions can also be regarded tautologically as coordinates on $(Z_t, \frac{1}{t}\omega_t)$, at least for small enough values of $t$, and morever $\frac{1}{t}\tilde{\omega}_t$ converges smoothly in these coordinates to $\omega_\infty$, after taking subsequence. If we fix a value of $u$ and restrict to the constant $u$ slice, then by construction, we are looking at the fibre metric $\frac{1}{t}\tilde{\omega}_t|_{X_{y=( f(P)+t^{1/2}u )} }$. Since locally $H$ is bounded below, we have the convergence estimate
\[
\frac{1}{t}\tilde{\omega}_t|_{X_{y} }=\omega_{SF,  y    }+O(\frac{1}{|\log t|^{1/2-\epsilon} }).
\]

But it is known from the work of Rong and Zhang (\cf \cite{Zhang}, theorem 1.4, and our comments \ref{RongandZhangsmoothconvergence}) that the family of Calabi-Yau metrics $\omega_{SF,  y    }$ on the smoothing of a Calabi-Yau variety converges to the CY metric on the central fibre $\omega_{SF,  f(P)    }$ away from the singular set, as $y$ tends to $f(P)$. So $\frac{1}{t}\tilde{\omega}_t|_{X_{y=( f(P)+t^{1/2}u )} } \to \omega_{SF,  f(P)    }$ as $t\to 0$, on each constant $u$ slice, as claimed.
\end{proof}

We also need information about the horizontal component of the metric. This essentially follows from the concentration estimate (proposition \ref{concentrationestimate}). Before a more formal discussion, we comment that the definition of $u$ requires an implicit choice of a local chart on the base $Y$, and the Euclidean metric $\omega_\C$ on the $u$-plane $\C$ needs to be specified by a scaling factor. We recall in the proof of the concentration estimate we introduced a Euclidean metric $\omega_{Y,y}$, corresponding to the fibre $X_y$. Now let us impose the normalising convention that under the pullback of the map \[
B(f(P),R)\subset Y \to \C, \quad y\mapsto t^{-1/2}(y-f(P)),
\]
the metric $\omega_\C$ on $\C$ pulls back to $   \frac{a_0}{nt}\omega_{Y,f(P)}$. 

\begin{prop}
The metric $\omega_\infty$ over the smooth part of $X_{f(P)}\times \C$ satisifies the Riemannian submersion property
\[
\Tr_{\omega_\infty} \omega_\C=1.
\]
\end{prop}

\begin{proof}
As above study a neighbourhood of a point $Q$ with $H\geq C$, and let $Q_t$ be approximations of $Q$ in the GH metric.
In the coordinates $z_1, z_2,\ldots, u$, the quantity $\Tr_{\frac{1}{t}    \tilde{\omega}_t} \omega_\C$ has smooth bounds. The concentration estimate \ref{concentrationestimate} can be interpreted as 
\[
\norm{ \Tr_{\frac{1}{t}    \tilde{\omega}_t} \omega_\C-1    }_{L^1(\{ |u-pr_\C f_t(Q_t) |\leq 1  \}, \frac{1}{t}\tilde{\omega}_t  )   } \leq \frac{C}{|\log t|}.
\]
In particular, upon passing to the $L^1$ limit, the trace $\Tr_{\omega_\infty} \omega_\C=1$. By the smooth convergence, in this coordinate we in fact have convergence to all orders.
\end{proof}

\begin{prop}
Over the smooth part of $X_{f(P)}\times \C$, The differential $du$ is parallel with respect to $\omega_\infty$, where $u$ is the coordinate on $\C$.
\end{prop}

\begin{proof}
Recall the gradient estimate (\ref{gradientestimateChernLu}), which can be written in a scaled version
\[
\int_{ |u|\leq  1   } (\frac{|\nabla du|_{\frac{1}{t} \tilde{\omega}_t}^2 }
{\Tr_{ \frac{1}{t} \tilde{\omega}_t  } \omega_\C  }
-|\partial \log \Tr_{ \frac{1}{t} \tilde{\omega}_t  } \omega_\C |_{ \frac{1}{t} \tilde{\omega}_t   }^2)
(\frac{1}{t} \tilde{\omega}_t)^n \leq \frac{C } {|\log t| }.
\]
Restricted to any compact set in the smooth locus, the trace term converges smoothly to the constant one, so in the limit
\[
\int |\nabla du|_{\omega_\infty}^2 \omega_\infty^n =0.
\]
It is clear that $u\leq 1$ is not an essential restriction. We get the qualitative statement that $\nabla du=0$ pointwise on the smooth locus.
\end{proof}

The parallel diffential $du$ induces a parallel $(1,0)$ type vector field by the complexified Hamiltonian construction:
\[
\iota_V \omega_{\infty}=d\bar{u}.
\]
This $V$ is holomorphic on the regular locus of $X_{f(P)}\times \C$. Now restrict  $V$ to the regular part of each fibre $X_{f(P)}$; this defines a holomorphic vector field in the bundle $TX_{f(P)}\oplus \C$. This $\C$ factor has generator $\frac{\partial }{\partial u}$.

We impose the technical condition that there is no nonzero holomorphic tangent vector field on $X_{f(P)}^{\text{reg}}$; it is satisfied for example if $X_{f(P)}$ is a nodal K3 surface. This condition forces $V$ to lie entirely in the $\C$ factor. Morever, on each fibre $X_{f(P)}$, $V$ must be a constant multiple of $\frac{\partial }{\partial u}$. We can then write $V=\lambda(u)  \frac{\partial }{\partial u}$, where $\lambda$ is a holomorphic function in $u$. Since $du$ and $V$ are both parallel, the quantity $\lambda=du(V)$ must be a constant. These give us enough information to show the limiting metric is a product.

\begin{prop}
The limiting metric $\omega_\infty=\omega_{SF,f(P)}+\omega_\C$ over the smooth locus of $X_{f(P)}\times \C$.
\end{prop}

\begin{proof}
We know $\omega_\infty$ restricted to the fibres is just $\omega_{SF,f(P)}$. By construction, the vector field $V$ defines the Hermitian orthogonal complement of the holomorphic tangent space of the fibres. We know $V$ is a constant multiple of $\frac{\partial }{\partial u}$, where the constant is specified by the Riemmanian submersion property. The claim follows.
\end{proof}

\begin{cor}
The metric space $(  X_{f(P)}\times \C  ,\omega_{SF,f(P)}+\omega_\C)$ is the metric completion of its smooth locus.
\end{cor}

\begin{proof}
This follows from the corresponding statement without the $\C$ factor.
This is true for Lefschetz fibrations with $n\geq 3$, for example by the recent work of Hein and Sun \cite{HeinSun}, the singular Calabi-Yau metric on the ODP fibre is modelled on the Stenzel metric.
\end{proof}

\begin{cor}
The holomorphic inverse map $f_\infty^{-1}$ defined on the smooth locus $(X_{f(P)}^{\text{reg} }\times \C, \omega_{SF,f(P)}+\omega_ \C)$ extends to a continuous map \[f_\infty^{-1}: (  X_{f(P)}\times \C  ,\omega_{SF,f(P)}+\omega_ \C)\to Z,\] which is a right inverse of $f_\infty$, and isometric to its image.  
\end{cor}

\begin{thm}
In the Lefschetz fibration case, suppose we assume the uniform fibre diameter bound, and that there is no nonzero holomorphic tangent vector field on $X_{f(P)}^{\text{reg}}$. Then		
the map $f_\infty: (Z, \omega_\infty) \to (X_{f(P)} \times \C, \omega_{SF,f(P)}+\omega_ \C)$ is an isometry, and naturally an isomorphism of normal varieties.
\end{thm}

\begin{proof}
The right inverse $f_\infty^{-1}$ must be surjective using the full measure property of the set $\{H>0 \}$ on $Z$. This combined with the above corollary proves the isometry statement, so in particular implies homeomorphism. We know $f_\infty$ is a biholomorphism \[\{ H>0\}\subset Z \simeq   X_{f(P)}^{\text{reg} }\times \C.\] 
Since a point in $Z$ with $H=0$ is a metric singularity, it cannot be a smooth point in the complex geometric sense, so $\{H>0\}\subset Z$ coincides with the regular locus $Z^\text{reg}$, and $f_\infty$ is an isomorphism between regular loci. This allows us to identify the structure sheaf of $Z$ with that of the normal variety $X_{f(P)}\times \C$.
\end{proof}

\section{Gromov Hausdorff convergence to the base}\label{GHconvergencetobase}

The collapsing sequence involves two apparent characteristic scales, where the fibre diameter is of constant order, or the base diameter is of constant order. We now turn to the latter case, and study the Gromov-Hausdorff convergence. The optimistic expectation is that there is a GH limit space homeomorphic to the base manifold $Y$, and equipped with a K\"ahler metric $\tilde{\omega}_0  $ on $Y\setminus S$, which extends to a K\"ahler current on $Y$, cohomologous to $\omega_Y$, solving the so called generalised K\"ahler-Einstein equation (\cf \cite{To} )
\[
\text{Ricci}( \tilde{\omega}_0 )=\text{Weil-Petersson metric},
\]
or rather its integrated form
\begin{equation}\label{generalisedKahlerEinstein}
\tilde{\omega}_0 ^m=\frac{ \int_Y  \omega_Y^m} {  \int_X i^{n^2}\Omega\wedge \overline{\Omega}  } f_*( i^{n^2}\Omega\wedge \overline{\Omega}    )    ,
\end{equation}
and morever $Y$ is homeomorphic to the metric completion of the smooth locus $Y\setminus S$.  This question has received a lot of attention since the work of Gross and Wilson \cite{GrossWilson}, and has been studied notably by \cite{To}, \cite{GrossToZhang}, \cite{GrossToZhang2}. Part of the difficulties come from the absence of a good Riemannian convergence theory in the collapsing case, and for higher dimensional base, the lack of knowledge of the limiting metric near the singular locus. V. Tosatti informs the author that this picture for one dimensional base $Y$ has been established in \cite{GrossToZhang2}, using Hodge theory and collapsed Cheeger-Colding theory. For higher dimensions, the main settled case is the hyperK\"ahler manifolds with a holomorphic Lagrangian torus fibration.

The aim of this section is a much easier proof of the isometric description of the GH limit in the following special situation.

\begin{thm}
Suppose the volume condition (\ref{volumecondition}) and the uniform fibre diameter bound (\ref{uniformfibrediameter}) hold, and the base manifold $Y$ is one dimensional, then the sequence of Ricci flat spaces $(X, \tilde{\omega}_t)$ converges in GH sense to $Y$ with the K\"ahler metric 
\begin{equation}
\tilde{\omega}_0=\frac{ \int_Y  \omega_Y} {  \int_X i^{n^2}\Omega\wedge \overline{\Omega}  } f_*( i^{n^2}\Omega\wedge \overline{\Omega}    ) =\frac{a_0}{n}  f_*( i^{n^2}\Omega\wedge \overline{\Omega}    ).
\end{equation}
\end{thm}

\begin{rmk}
We observe that there is a natural comparison map $X\to Y$. It is easy to write down the expected limit for one dimensional base, because on Riemann surfaces the K\"ahler form is the volume form, which is prescribed by the generalised K\"ahler Einstein equation. Here $a_0$ is just a cohomological constant in the normalisation of the volume constant (compare (\ref{CalabiYaucondition})). To link this to the earlier part of this work, recall the locally defined Euclidean metric $\omega_{Y,y}$ has the normalisation $\omega_{Y,y}|_y=f_*( i^{n^2}\Omega\wedge \overline{\Omega}    )|_y$. (\cf proposition \ref{concentrationestimate}) Notice also that the volume condition (\ref{volumecondition}) implies $\tilde{\omega_0}$ as defined above is uniformly equivalent to the smooth metric $\omega_Y$ on $Y$. It is also clear that $\tilde{\omega_0}$ is smooth away from $S$. It is not smooth on $S$, but under our assumptions it is uniformly equivalent to the smooth K\"ahler metric $\omega_0$.	
\end{rmk}


We first show a distance decreasing behaviour which can be compared to Yau's Schwarz lemma.

\begin{prop}(Almost Schwarz lemma)
The derivative of the map $f: (X, \tilde {\omega_t})\to (Y, \tilde{\omega_0})$ satisfies $|\nabla f|_{L^\infty}\leq 1+\frac{C}{|\log t|}$, and therefore
\begin{equation}
d_{ \tilde{\omega_0}   }(f(x_1), f(x_2 ) )\leq d_{ \tilde {\omega_t}  }(x_1, x_2)  (1+\frac{C}{|\log t|}   )
\end{equation}
\end{prop}

\begin{proof}
This is a simple reinterpretation of the concentration estimate (proposition \ref{concentrationestimate}), using $|\nabla f|^2= \Tr_{ \tilde {\omega_t}   } \tilde {\omega_0} $.
\end{proof}

Now we wish to show a bound in the converse direction. Fix a small $\epsilon>0$.

\begin{lem}
If $\gamma$ is a path of length $l(\gamma)$ in $(Y, \tilde{\omega_0} )$ which stays distance $\epsilon$ bounded away from the singular locus $S$, then $\gamma$ admits a lift to the total space $(X, \tilde {\omega_t}) $ which has length at most $l(\gamma)(1+C(\epsilon)/|\log t|)$. 
\end{lem}

\begin{proof}
We recall the smooth bound (\ref{smoothboundTromega0}). When combined with the concentration bound on the trace (proposition \ref{concentrationestimate}), it says $\Tr_{ \tilde {\omega_t}   } {\omega_0}$ can oscillate on the fibre from its average by at most $\frac{C}{|\log t|}$, and the same can be said about $\Tr_{ \tilde {\omega_t}   } {\tilde{\omega_0} }$, which only differs from $\Tr_{ \tilde {\omega_t}   } {\omega_0}$ by a bounded multiple constant on any fixed fibre. This means in particular that on smooth fibres bounded away from $S$,
\[
|\Tr_{ \tilde {\omega_t}   } {\tilde{\omega_0} }-1| \leq \frac{C}{|\log t|}.
\]
The interpretation of this is that $f$ is almost a Riemannian submersion. Whenever there is a Riemannian metric on the total space, the tangent space on $X$ splits into the vertical and the horizontal parts, which defines an Ehresmann connection. The above inequality means the horizontal lift of any unit tangent vector on $Y$ is also a unit tangent vector up to $O(1/|\log t|)$ error. From this the claim is clear.
\end{proof}

\begin{cor}
If $x, x'$ are two points in $X$, then 
\begin{equation}
d_{ \tilde {\omega_t} }(x,x') \leq ( 1+   \frac{C}{|\log t|}  ) d_{ \tilde {\omega_0}  }(x, x')+  C\epsilon +Ct^{1/2}   
\end{equation} 
\end{cor}

\begin{proof}
Take a curve $\gamma$ joining $f(x)$, $f(x')$, which almost minimises the length: $l(\gamma) \leq d_{ \tilde {\omega_0}  }(f(x), f(x'))+\epsilon $. Then $\gamma$ can spend at most $C\epsilon$ length in the $\epsilon$ neighbourhood of any point in $S$. Observe also that $S$ is a finite set for one dimensional base. Now the diameter of the preimage of the $\epsilon$ neighbourhood of any singular point in $S$ is bounded by $C\epsilon+Ct^{1/2}$, using the fibre diameter bound $Ct^{1/2}$. For the part of $\gamma$ with distance at least $\epsilon$ away from $S$, we can use the previous lemma to construct a lift, and the claim follows.
\end{proof}

From this we see $f$ is almost an isometry, up to an error which can be made smaller than $C\epsilon$ for small $t$. This implies we can define a GH metric on the disjoint union of $X$ and $Y$, which restricts to the distance function on $X\times X$ and $Y\times Y$ individually, and on $X\times Y$,
$
d(x,y)=d_{\tilde{\omega_0} }(f(x), y)+C\epsilon.
$
Both $X$ and $Y$ are $C\epsilon$ dense, so $X$ and $Y$ must be $C\epsilon$ close in Gromov Hausdorff metric. From this the theorem is shown.

\end{document}